\documentclass{article}

\usepackage[final, nonatbib]{neurips_2020}

\usepackage[utf8]{inputenc} 
\usepackage[T1]{fontenc}    
\usepackage{url}            
\usepackage{booktabs}       
\usepackage{amsfonts}       
\usepackage{nicefrac}       
\usepackage{microtype}      
\usepackage{amsmath}
\usepackage[dvipsnames]{xcolor}
\usepackage{verbatim}

\usepackage{graphicx}
\usepackage{amsthm}
\usepackage{amsmath}
\usepackage{amssymb}
\usepackage{wrapfig}
\usepackage{algorithm}
\usepackage{algorithmic}
\usepackage[algo2e, noend, noline, ruled, linesnumbered]{algorithm2e}

\usepackage{mathtools}
\mathtoolsset{showonlyrefs=true}

\usepackage{tikz}
\usepackage{caption}
\usepackage{subcaption}
\usepackage{selectp}

\newcommand{\bW}{\mathbf{W}}
\newcommand{\BR}{\text{BR}}

\newtheorem{property}{Property}
 
\newtheorem{theorem}{Theorem}

\theoremstyle{definition}
 
\theoremstyle{remark}

\title{Fictitious Play for Mean Field Games: \\Continuous Time Analysis and Applications}

\author{%
  Sarah Perrin\thanks{Equal contribution} $^{\; , 1}$, \; Julien Perolat\footnotemark[1] $^{\; , 2}$, \; Mathieu Lauri\`ere$^{3}$, \; Matthieu Geist$^{4}$,\AND
  Romuald Elie$^{2}$, \; Olivier Pietquin$^{4}$ \vspace{0.7em} \\
  Univ. Lille, CNRS, Inria, UMR 9189 CRIStAL$^{1}$\hspace{1cm} DeepMind Paris$^{2}$ \vspace{0.5em} \\
  Princeton University, ORFE$^{3}$\hspace{1cm} Google, Research Brain Team$^{4}$ \vspace{0.7em} \\
  \texttt{sarah.perrin@inria.fr}\hspace{1cm}\texttt{perolat@google.com} \vspace{0.5em} \\
  \texttt{lauriere@princeton.edu}\hspace{1cm}\texttt{[mfgeist, relie, pietquin]@google.com}
}

\begin{document}

\maketitle

\begin{abstract}
  In this paper, we deepen the analysis of continuous time Fictitious Play learning algorithm to the consideration of various finite state Mean Field Game settings (finite horizon, $\gamma$-discounted), allowing in particular for the introduction of an additional common noise. 
  We first present a theoretical convergence analysis of the continuous time Fictitious Play process and prove that the induced  exploitability decreases at a rate $O(\frac{1}{t})$. Such analysis emphasizes the use of exploitability as a relevant metric for evaluating the convergence towards a Nash equilibrium in the context of Mean Field Games. These theoretical contributions are supported by numerical experiments provided in either model-based or model-free settings. We provide hereby for the first time converging learning dynamics for Mean Field Games in the presence of common noise. 
\end{abstract}

\section{Introduction}
Learning in games has a long history~\cite{Shan50,samuel1959some} but learning in the midst of a large number of players still remains intractable. Even the most recent successes of machine learning, including Reinforcement Learning (RL)~\cite{Sutton2018}, remain limited to interactions with a handful of players (\textit{e.g.} Go~\cite{silver2016mastering,silver2017mastering, Silver18AlphaZero}, Chess~\cite{campbell2002deep}, Checkers~\cite{schaeffer2007checkers, samuel1959some}, Hex~\cite{Anthony17ExIT}, Starcraft II~\cite{vinyals2019grandmaster}, poker games~\cite{Brown17Libratus, Brown19Pluribus, moravvcik2017deepstack, bowling2015heads} or Stratego \cite{mcaleer2020pipeline}). Whilst the general multi-agent learning case might seem out of reach, considering interactions within a very large population of players may lead to tractable models. Inspired by the large economic literature on games with a continuum of players~\cite{aumann1964markets}, the notion of Mean Field Games (MFGs) has been introduced in \cite{MR2295621, MR2346927} to model strategic interactions through the distribution of players' states. In such framework, all players are identical, anonymous (\textit{i.e.}, they are not identifiable) and have symmetric interests. In this asymptotic formulation, the learning problem can be reduced to characterizing the optimal interactions between one representative player and the full population.

Most of the MFG literature assumes the representative player to be fully informed about the game dynamics and the associated reward mechanisms. In such context, the Nash equilibrium for an MFG is usually computed via the solution of a coupled system of dynamical equations. The first equation models the forward dynamics of the population distribution, while the second is the dynamic programming equation of the representative player. Such approaches typically rely on partial differential equations and require deterministic numerical approximations~\cite{achdoulauriere2020mfgnumerical} (\textit{e.g.}, finite differences methods~\cite{MR2679575,achdou2012mean}, semi-Lagrangian schemes~\cite{MR3148086,MR3392626}, or primal-dual methods~\cite{MR3772008,BricenoAriasetalCEMRACS2017}). Despite the success of these schemes, an important pitfall for applications is their lack of scalability. In order to tackle this limitation, stochastic methods based on approximations by neural network have recently been introduced in~\cite{CarmonaLauriere_DL_periodic,CarmonaLauriere_DL,fouque2019deep} using optimality conditions for general mean field games, in~\cite{ruthotto2020machine} for MFGs which can be written as a control problem, and in~\cite{cao2020connecting,lin2020apac} for variational MFGs in connection with generative adversarial networks. We now contribute and take a new step forward in this direction.

We investigate a generic and scalable simulation-based learning algorithm for the computation of approximate Nash equilibria, building upon the Fictitious Play scheme~\cite{robinson1951iterative, fudenberg1998theory, shapiro1958}.
We study the convergence of Fictitious Play for MFGs, using tools from the continuous learning time analysis~\cite{harris1998rate, ostrovski2013payoff, hofbauer2002global}. We derive a convergence of the Fictitious Play process at a rate $O(\frac{1}{t})$ in finite horizon or over $\gamma$-discounted monotone MFGs (see Appx. \ref{Sec_LQ_Appx}), thus extending previous convergence results restricted to simpler games~\cite{harris1998rate}. Besides, our approach covers games where the players share a common source of risk, which are widely studied in the MFG literature and crucial for applications.
To the best of our knowledge, we derive for the first time convergence properties of a learning algorithm for these so-called MFGs with common noise (where a common source of randomness affects all players~\cite{carmona2018probabilisticI-II}).  Furthermore, our analysis emphasizes the role of  \textit{exploitability} as a relevant metric for characterizing the convergence towards a Nash equilibrium, whereas most approximation schemes in the MFG literature quantify the rate of convergence of the population empirical distribution.
The contribution of this paper is thus threefold: (1)  we provide several theoretical results concerning the convergence of \emph{continuous time} Fictitious Play in MFGs matching the $O(\frac{1}{t})$ rate existing in zero-sum two-player normal form game, (2) we generalize the notion of \emph{exploitability} to MFGs and we show that it is a meaningful metric to evaluate the quality of a learned control in MFGs, and (3) we empirically illustrate the performance of the resulting algorithm on several MFG settings, including examples with \emph{common noise}.
\section{Background on Finite Horizon Mean Field Games}
A Mean Field Game (MFG) is a temporally extended decision making problem involving an infinite number of identical and anonymous players. It can be solved by focusing on the optimal policy of a representative player in response to the behavior of the entire population. Let $\mathcal{X}$ and $\mathcal{A}$ be finite sets representing respectively the state and action spaces. The representative player starts the game in state $x \in \mathcal{X}$ according to an initial distribution $\mu_0$ over $\mathcal{X}$. At each time step $n \in [0,\dots, N]$, the representative player being in state $x_n$ takes an action $a_n$ according to a policy $\pi_n(a_n|x_n)$. As a result, the player moves to state $x_{n+1}$ according to the transition probability $p(.|x_n, a_n)$ and receives a reward $r(x_n, a_n, \mu_n)$, where $\mu_n$ represents the distribution over states of the entire population at time $n$. 
For a given sequence of policies $\pi = (\pi_n)_n$ and a given sequence of distributions $\mu=(\mu_n)_n$, the representative player will receive the cumulative sum of rewards defined as\footnote{All the theory can be easily extended in the case where the reward is also time dependent.}:
\begin{align}
    J(\mu_0, \pi, \mu) = \mathbb{E}\left[\sum_{n=0}^N r(x_n, a_n, \mu_n) \; | \; x_0 \sim \mu_0, \; x_{n+1} = p(.|x_n, a_n), \; a_n \sim \pi_n(.|x_n)\right].
\end{align}
\textbf{$Q$-functions and value functions:}
The $Q$-function is defined as the expected sum of rewards starting from state $x$ and doing action $a$ at time $n$:
\begin{align}
    & Q^{\pi, \mu}_n(x, a) 
    = \mathbb{E}\left[\sum_{k=n}^N r(x_k, a_k, \mu_k) \; | \; x_n = x,\; a_n=a, \; x_{k+1} = p(.|x_k, a_k), \; a_k \sim \pi_k(.|x_k)\right].
\end{align}
By construction, it satisfies the recursive equation:
$$
    Q^{\pi, \mu}_N(x,a) = r(x, a, \mu_N), 
    \quad Q^{\pi, \mu}_{n-1}(x,a) = r(x, a, \mu_{n-1}) + \sum \limits_{x' \in \mathcal{X}}p(x'|x,a)\mathbb{E}_{b \sim \pi_{n}(.|x')}\left[Q^{\pi, \mu}_{n}(x', b)\right].
$$
The value function is the expected sum of rewards for the player that starts from state $x$ and can thus be defined as: $V_n^{\pi,\mu}(x)=\mathbb{E}_{a \sim \pi(.|x)}\left[Q^{\pi, \mu}_n(x, a)\right]$.
Note that the objective function $J$ of a representative player rewrites in particular as an average at time $0$ of the value function $V$ under the initial distribution $\mu_0$: $J(\mu_0, \pi, \mu) = \mathbb{E}_{x \sim \mu_0(.)} \left[V^{\pi, \mu}_0(x)\right].$

\textbf{Distribution induced by a policy:}
The state distribution induced by $\pi = \{\pi_n\}_n$ is defined recursively by the forward equation starting from $\mu^\pi_0(x) = \mu_0(x)$ and $\mu^\pi_{n+1}(x') = \sum \limits_{x, a \in \mathcal{X}\times \mathcal{A}} \pi_n(a|x) p(x'|x,a)\mu^\pi_{n}(x)$.

\textbf{Best Response:} A best response policy $\pi^{BR}$ is a policy that satisfies $ J(\mu_0, \pi^{BR}, \mu^{\pi}) = \max \limits_{\pi'} J(\mu_0, \pi', \mu^{\pi})$. Intuitively, it is the optimal policy an agent could take if it was to deviate from the crowd's policy.

\textbf{Exploitability:} The exploitability $\phi(\pi)$ of policy $\pi$ quantifies the average gain for a representative player to replace its policy by a best response, while the entire population plays with policy $\pi$:
$\phi(\pi) := \max \limits_{\pi'} J(\mu_0, \pi', \mu^{\pi}) - J(\mu_0, \pi, \mu^{\pi})$. Note that, as it scales with rewards, the absolute value of the exploitability is not meaningful. What matters is its relative value compared with a reference point, such as the exploitability of the policy at initialization of the algorithm. In fact, the exploitability is game dependent and hard to re-scale without introducing other issues (dependence on the initial policy if we re-normalize with the initial exploitability for example). 

\textbf{Nash equilibrium:} A Nash equilibrium is a policy satisfying $\phi(\pi)=0$ while an approximate Nash equilibrium has a small level of exploitability. 

The exploitability is an already well known metrics within the computational game theory literature~\cite{zinkevich2008regret, bowling2015heads, lanctot2009monte, burch2014solving}, and one of the objectives of this paper is to  emphasize its important role in the context of MFGs. Classical ways of evaluating the performance of numerical methods in the MFG literature typically relate to distances between distribution $\mu$ or value function $V$, as for example in~\cite{achdoulauriere2020mfgnumerical}. A close version of the exploitability has been used in this context (\textit{e.g.},~\cite{guo2019learning}), but being  computed over all possible starting states at any time. Such formulation gives too much importance to each state, in particular those having a (possibly very) small probability of appearance. In comparison, the exploitability provides a  well balanced average metrics over the trajectories of the state process.

\textbf{Monotone games:}
A game is said monotone if the reward has the following structure:
$r(x,a,\mu) = \tilde r(x,a) + \bar r(x,\mu)$ and $ \forall \mu, \mu', \; \sum_{x \in \mathcal{X}} (\mu(x) - \mu'(x))(\bar r(x,\mu) - \bar r(x,\mu'))\leq 0.$
 This so-called Lasry-Lions monotonicity condition is classical to ensure the uniqueness of the Nash equilibrium~\cite{MR2295621}.

\textbf{Learning in finite horizon problems:} When the distribution $\mu$ of the population is given, the representative player faces a classical finite horizon Markov Decision problem. Several approaches can be used to solve this control problem such as model-based algorithms (\textit{e.g.} backward induction: Algorithm~\ref{algBI} in Appx.~\ref{Sec_Algorithm}, with update rule $\forall a, x \in \mathcal{A} \times \mathcal{X} \; Q^{\mu}_{n-1}(x,a) = r(x, a, \mu_{n-1}) + \sum_{x' \in \mathcal{X}}p(x'|x,a)\max \limits_b Q^{\mu}_{n}(x', b)$) or model-free algorithms (\textit{e.g.} $Q$-learning: Algorithm~\ref{algQ_learning} in Appx.~\ref{Sec_Algorithm} with update rule $Q^{k+1}_n(x^k_n, a_n^k) = (1-\alpha)Q^{k+1}_n(x^k_n, a_n^k) + \alpha[r(x^k_n, a_n^k, \mu_{k-1}) + \max_b Q^{k}_{n+1}(x^k_{n+1}, b)]$). 

\textbf{Computing the population distribution:}
Once a candidate policy is identified, one needs to be able to compute (or estimate) the induced distribution of the population at each time step. It can either  be computed exactly using a model-based method such as Algorithm~\ref{algDE} in Appx. \ref{Sec_Algorithm},  or alternatively be estimated with a model-free method like Algorithm~\ref{algEDE} in Appx. \ref{Sec_Algorithm}.

\textbf{Fictitious Play for MFGs:}
 Consider available (1) a computation scheme for the  population distribution given a policy, and (2) an approximation algorithm for an optimal policy of the representative player in response to a population distribution. Then, discrete time Fictitious Play presented in Algorithm~\ref{algFP} provides a robust  approximation scheme for Nash equilibrium by computing iteratively the best response against the  distribution induced by the average of the past best responses. We will analyse this discrete time process in continuous time in section~\ref{CTFP_in_MFG}. To differentiate the discrete time from the continuous time, we denote the discrete time with $j$ and the continuous time with $t$. At a given step $j$ of Fictitious Play, we have that:
\begin{equation}
\label{eq:discrete-bar-mu}
\forall n,\; \bar \mu^j_n= \frac{j-1}{j}\bar \mu^{j-1}_n + \frac{1}{j}\mu^{\pi^j}_n
\end{equation}
The policy generating this average distribution is:
\begin{equation}
\label{eq:discrete-bar-pi-nj}
    \forall n,\; \bar \pi^j_n(a|x) = \frac{\sum _{i=0}^j \mu^{\pi^i}_n(x) \pi^i_n(a|x)}{\sum_{i=0}^j \mu^{\pi^i}_n(x)}.
\end{equation}
\begin{algorithm2e}[ht!]
\SetKwInOut{Input}{input}\SetKwInOut{Output}{output}
\Input{Start with an initial policy $\pi_0$, an initial distribution $\mu_0$ and define $\bar{\pi}_0 = \pi_0$}
\caption{Fictitious Play in Mean Field Games \label{algFP}}
\For{$j=1,\dots,J$:}{
      find $\pi^j$ a best response against $\bar \mu^j$ (either with $Q$-learning or with backward induction)\;
      compute $\bar\pi^{j}$ the average of $(\pi^0, \dots, \pi^{j})$\;
      compute $\mu^{\pi^{j}}$ (either with a  model-free or model-based method)\;
      compute $\bar \mu^j$ the average of $(\mu^0, \dots, \mu^{\pi^{j}})$  
}
\Return{$\bar \pi^J$, $\bar \mu^J$}
\end{algorithm2e}

\section{Continuous Time Fictitious Play in Mean Field Games}
\label{CTFP_in_MFG}

In this section, we study a continuous time version of Algorithm~\ref{algFP}. The continuous time Fictitious Play process is defined  following the lines of \cite{harris1998rate, ostrovski2013payoff}.
First, we start for $t<1$ with a fixed policy $\bar \pi ^{t<1} = \{\bar\pi_{n}^{t<1}\}_{n} = \{\pi_{n}^{t<1}\}_{n}$ with induced distribution  $\bar \mu^{t<1} = \mu^{t<1}=\mu^{\pi^{t<1}} = \{\mu^{\pi^{t<1}}_n \}_n$ (this arbitrary policy for $t\in[0,1]$ is necessary for the process to be defined at the starting point). Then, the Fictitious Play process is defined for all $t\geq 1$ and $n \in [1,\dots ,N]$ as:
\begin{align}
    \frac{d}{dt} \bar \mu_n^t (x) = \frac{1}{t}\left(\mu_n^{\BR, t} (x)-\bar \mu_n^t (x)\right) \quad
    \textrm{ or in integral form: }
    \quad \bar \mu_n^t (x) = \frac{1}{t} \int \limits_{s=0}^t \mu_n^{\BR, s} (x) ds\,,
\end{align}
where $\mu_n^{\BR, t}$ denotes the distribution induced by a best response policy $\{\pi^{\BR,t}_n\}_n$ against $\bar\mu_n^t (x)$. Hence, the distribution $\mu_n^t (x)$ identifies to the population distribution induced by the averaged policy $\{\pi_{n}^{t}\}_{n}$ defined as follows (proof in~\ref{app:CFP_FH}):
\begin{align}
    &\forall n, \;\bar\mu_n^t(x) \frac{d}{dt} \bar\pi_n^t(a|x)  = \frac{1}{t} \mu^{\BR,t}_n(x) [\pi^{\BR,t}_n(a|x) - \bar\pi_n^t(a|x)]\\
    &\textrm{ or in integral form: } \forall n, \;\bar\pi_n^t(a|x) \int \limits_{s=0}^t \mu^{\BR,s}_n(x) ds = \int \limits_{s=0}^t \mu^{\BR,s}_n(x) \pi^{\BR,s}_n(a|x) ds,
\end{align}
with $\pi^{\BR,s}_n$ being chosen arbitrarily for $t \leq 1$. We are now in position to provide the main result of the paper quantifying  the convergence rate of the continuous Fictitious Play process. 
\begin{theorem}
\label{thm:fp_FH}
If the MFG satisfies the  monotony assumption, we can show that the exploitability is a strong Lyapunov function of the system, $\forall t \geq 1$:
$\frac{d}{dt} \phi(\bar\pi^t) \leq - \frac{1}{t} \phi(\bar\pi^t).$
Hence $\phi(\bar\pi^t) = O(\frac{1}{t})$.
\end{theorem}
The proof of the theorem is postponed to Appendix~\ref{app:CFP_FH}. Furthermore, a similar property for $\gamma$ discounted MFGs is provided in Appendix \ref{Section_Gamma}. We chose to present an analysis in continuous time because it provides convenient mathematical tools allowing to exhibit  state  of  the  art  convergence  rate. In discrete time, similarly to normal form games \cite{karlin1959mathematical, daskalakis2014counter}, we conjecture that the convergence rate for monotone MFGs is $O(t^{-\frac{1}{2}})$.

\section{Experiments on Fictitious Play in the Finite Horizon Case}
\label{sec:expe-FP-FH}
In this section, we illustrate the theoretical convergence of continuous time Fictitious Play by looking at the discrete time implementation of the process. We focus on classical linear quadratic games which have been extensively studied~\cite{bensoussan2016linear,graber2016linear,duncan2018linear} and for which a closed form solution is available. We then turn to a more difficult numerical setting for experiments\footnote{In all experiments, we represent $\bar \mu$, but applying $\bar \pi$ to $\mu_0$ would give the same result as $\bar \mu = \mu^{\bar\pi}$.}. We chose either a full model-based implementation or a full model-free approach of Alg. \ref{algFP}. The model-based uses Backward Induction (Alg. \ref{algBI}) and an exact calculation of the population distribution (Alg. \ref{algEDE}). The model-free approach uses $Q$-learning (Alg. \ref{algQ_learning}) and a sampling-based estimate of the distribution (Alg.~\ref{algDE}). 

\subsection{Linear Quadratic Mean Field Game}
\textbf{Environment:} We consider a Markov Decision Process a finite action space $\mathcal{A} = \{-M, \dots, M\}$ together with a one dimensional finite state space domain $\mathcal{X} = \{-L, \dots, L\}$, which can be viewed as  a truncated and discretized version of $\mathbb{R}$. 
The dynamics of a typical player picking action $a_n$ at time $n$ are governed by the following equation:
\begin{align}
    x_{n+1} = x_{n} + (K(m_n - x_n) + a_n)\Delta_n + \sigma\epsilon_{n}\sqrt{\Delta_n}\;,
\end{align}
allowing the representative player to either stay still or move to the left or to the right. In order to make the model more complex, an additional discrete noise $\epsilon_n$ can also push the player to the left or to the right with a small probability: $\epsilon_{n} \sim \mathcal{N}(0,1)$, which is in practice discretized over $\{-3\sigma,\dots,3\sigma\}$. The resulting state $x_{n+1}$ is rounded to the closest discrete state.

At each time step, the player can move up to $M$ nodes and it receives the reward:
$$
    r(x_n,a_n,\mu_n) =  [ -\frac{1}{2}|a_n|^2 + q a_n (m_n - x_n) - \frac{\kappa}{2}(m_n - x_n)^2 ]\Delta_n
$$
where $m_n = \sum_{x \in \mathcal{X}} x\mu_n(x)$ is the first moment of the state distribution $\mu_n$. $\Delta_n$ is the time lapse between two successive steps, while $q$ and $\kappa$ are given non-negative constants. The first term quantifies the action cost,  while the two last ones encourage the player to remain close to the  average state of the population at any time.
Hereby, the optimal policy pushes each player in the direction of the population average state. We set the terminal reward to $r(x_N,a_N,\mu_N) =  -\frac{c_{\rm term}}{2}(m_N - x_N)^2$.

\begin{figure}[htbp]
    \centering
    \begin{subfigure}{0.25\textwidth}
      \centering
      \includegraphics[width=1.0\linewidth]{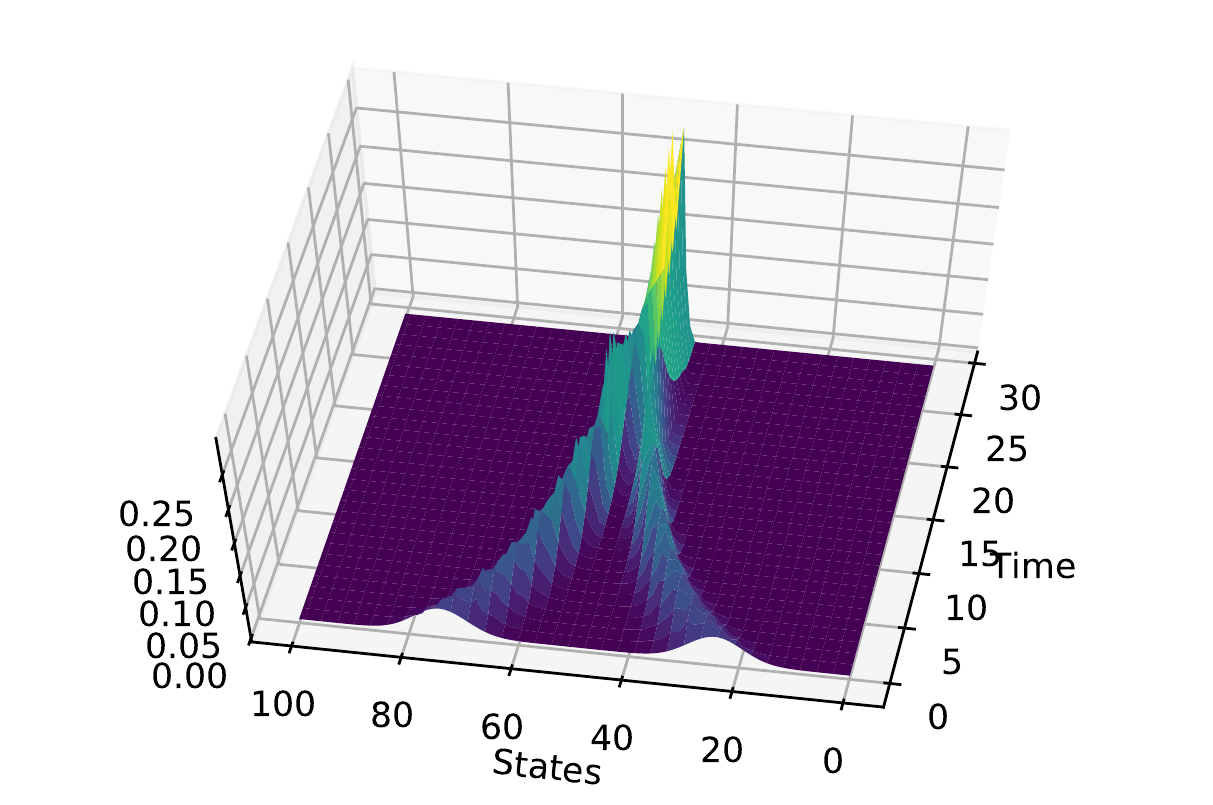}
      \caption{Exact Solution}
    \end{subfigure}%
    \begin{subfigure}{0.25\textwidth}
      \centering
      \includegraphics[width=1.0\linewidth]{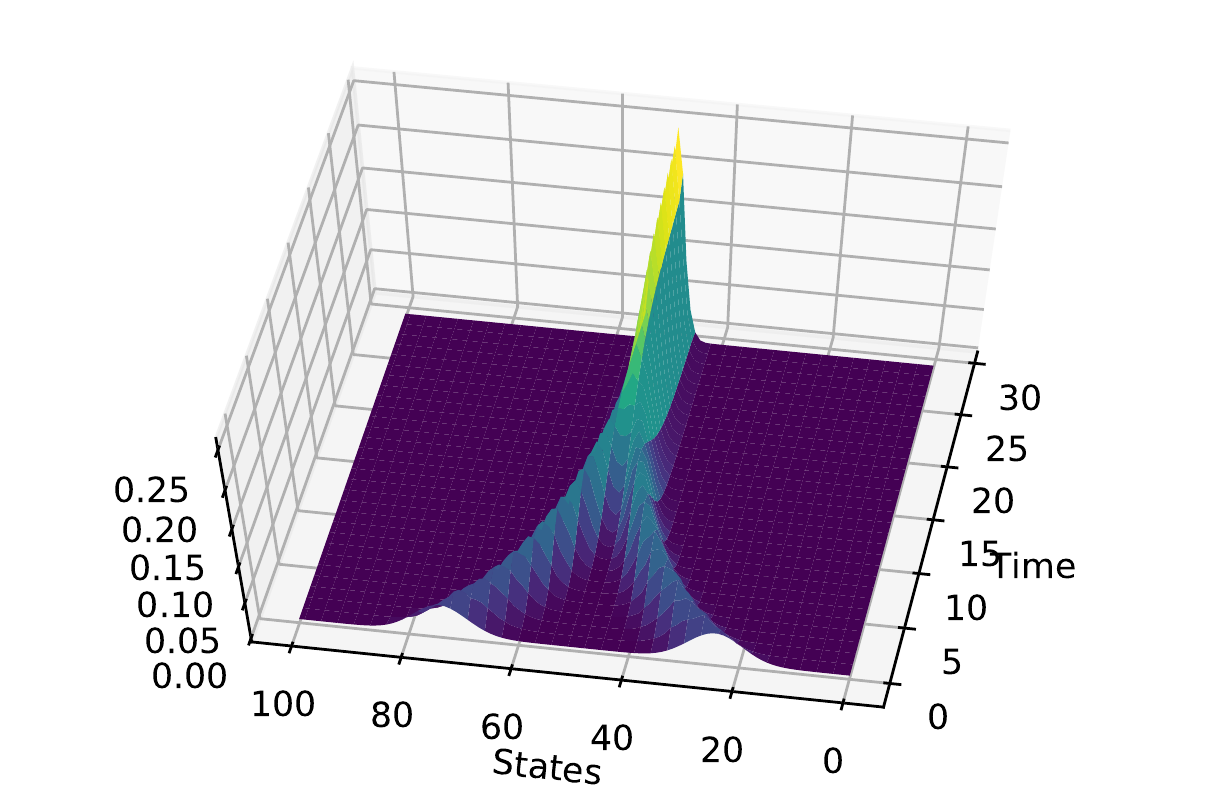}
      \caption{Model-based}
    \end{subfigure}%
    \begin{subfigure}{0.25\textwidth}
      \centering
      \includegraphics[width=1.0\linewidth]{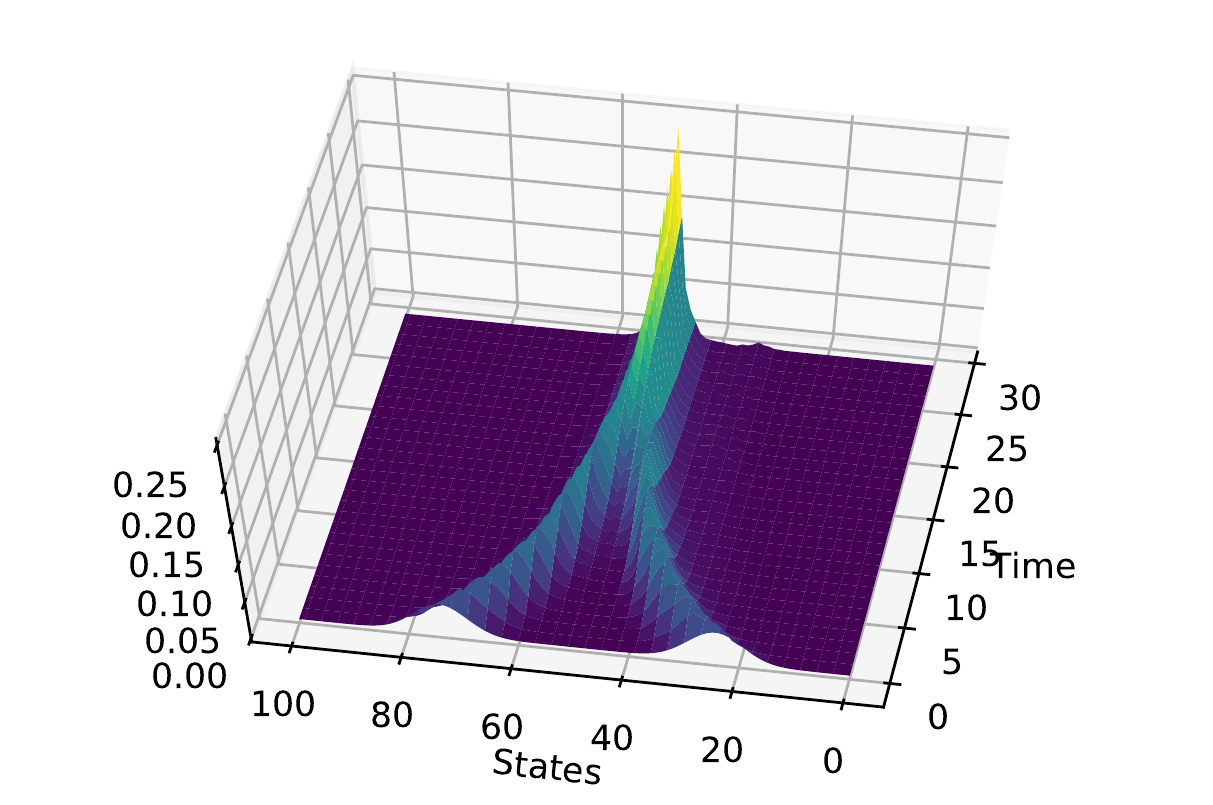}
      \caption{Model-free}
    \end{subfigure}%
    \begin{subfigure}{0.25\textwidth}
      \centering
      \includegraphics[width=1.0\linewidth]{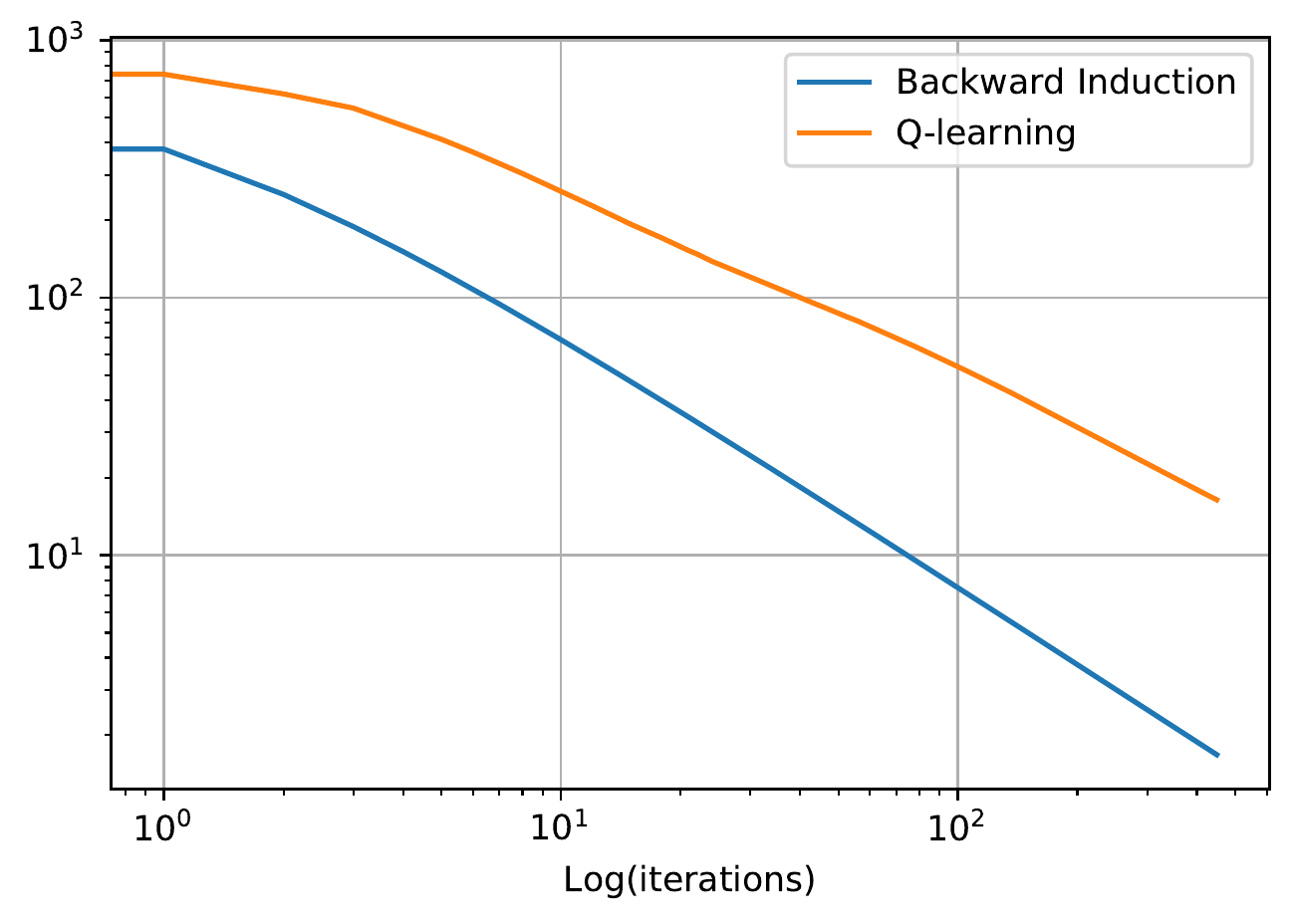}
      \caption{Exploitability}
    \end{subfigure}

    \caption{Evolution of the distribution in the linear quadratic MFG with finite horizon.}
    
    \label{fig:LQ}
    \end{figure}
    
\textbf{Experimental setup:} We consider a Linear Quadratic MFG with $100$ states and an horizon $N=30$, which provides a closed-form solution for the continuous state and action version of the game (see Appx. \ref{Section_Gamma}) and bounds the number of actions $M=37$ required in the implementation. 
In practice, the variance $\sigma$ of the idiosyncratic noise $\epsilon_n$ is adapted to the number of states. Here, we set $\sigma = 3$, $\Delta_n=0.1$, $K = 1$, $q = 0.01$, $\kappa = 0.5$ and $c_{\rm term}=1$. In all the experiments, we set the learning rate $\alpha$ of $Q$-learning to 0.1 and the $\varepsilon$-greedy exploration parameter to $0.2$.

\textbf{Numerical results:} Figure \ref{fig:LQ} illustrates the convergence of Fictitious Play model-based and model-free algorithm in such context. The initial distribution, which is set to two separated bell-shaped distributions, are both driven towards $m$ and converge to a unique bell-shaped distribution as expected. The parameter $\sigma$ of the idiosyncratic noise influences the variance of the final normal distribution. We can observe that both Backward Induction and $Q$-learning provide policies that approximate this behaviour, and that the exploitability decreases with a rate close to $O(1/t)$ in the case of the model-based approach, while the model-free decreases more slowly.

\subsection{The Beach Bar Process}

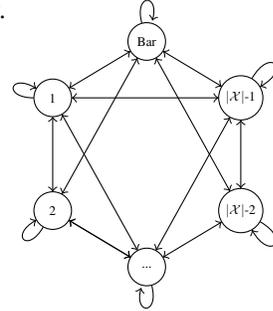
\begin{wrapfigure}{r}{0.35\textwidth}
    \vspace{-60pt}
    \centering
    \begin{tikzpicture}
        [shorten >=0pt, auto,
        node distance=2.5cm, scale=0.5, 
        transform shape, align=center, 
        state/.style={circle, draw, minimum size=1cm}]
        \node[state] (A) at (0,0) {2};
        \node[state] (B) at (0,3) {1};
        \node[state] (C) at (2.5,4.5) {Bar};
        \node[state] (D) at (5,0) {$|\mathcal{X}|$-2};
        \node[state] (E) at (2.5,-1.5) {...};
        \node[state] (F) at (5,3) {$|\mathcal{X}|$-1} ;
    
        \path [<->] (A) edge node[bend left] {} (B);
        \path [<->](B) edge node[left] {} (C);
        \path [<->](A) edge node[right] {} (E);
        \path [<->](C) edge node[right] {} (F);
        \path [<->](D) edge node[right] {} (F); 
        \path [<->](E) edge node[right] {} (D); 
        
        \path [<->] (A) edge node[left] {} (C);
        \path [<->](B) edge node[left] {} (E);
        \path [<->](A) edge node[right] {} (E);
        \path [<->](C) edge node[right] {} (D);
        \path [<->](B) edge node[right] {} (F); 
        \path [<->](E) edge node[right] {} (F); 
        
        \path [->](A) edge [out=240,in=210,looseness=8] node {} (A);
        \path [->](B) edge [out=180,in=150,looseness=8] node {} (B);
        \path [->](C) edge [loop above] node {} (C);
        \path [->](D) edge [out=330,in=300,looseness=8] node {} (D);
        \path [->](E) edge [loop below] node {} (E);
        \path [->](F) edge [out=60,in=30,looseness=8] node {} (F);
    \end{tikzpicture}
    \vspace{-2pt}
    \caption{The beach bar process.}
    \label{fig:beachbar}
    \vspace{-5pt}
\end{wrapfigure}

As a second illustration, we now consider the beach bar process, a more involved monotone second order MFG with discrete state and action spaces, that does not offer a closed-form solution but can be analyzed intuitively. This example is a simplified version of the well known Santa Fe bar problem, which has received a strong interest in the MARL community, see e.g.~\cite{arthur1994inductive,farago2002fair}.

\textbf{Environment:} The beach bar process (Figure~\ref{fig:beachbar}) is a Markov Decision Process with $|\mathcal{X}|$ states disposed on a one dimensional torus ($\mathcal{X} = \{0,\dots,|\mathcal{X}|-1\}$), which represents a beach. A bar is located in one of the states. As the weather is very hot, players want to be as close as possible to the bar, while keeping away from too crowded areas. Their dynamics is governed by the following equation: $$x_{n+1} = x_{n} + b(x_n, a_n) + \epsilon_{n}$$
where $b$ is the drift, allowing the representative player to either stay still or move one node to the left or to the right. The additional noise $\epsilon_n$ can push the player one node away to the left or to the right with a small probability:
\begin{align}
    b(x_n, a_n) = \left\{
    \begin{array}{ll}
        1 & \mbox{ if }  a_n = \mbox{right} \\
        0 & \mbox{ if }  a_n = \mbox{still} \\
        -1 & \mbox{ if } a_n = \mbox{left}
    \end{array}
\right. 
&&
\epsilon_{n} = \left\{
    \begin{array}{ll}
        1 & \mbox{ with probability }  \frac{1-p}{2} \\
        0 & \mbox{ with probability }  p \\
        -1 & \mbox{ with probability }  \frac{1-p}{2}
    \end{array}
\right.
\end{align}

Therefore, the player can go up to two nodes right or left and it receives, at each time step, the reward:
$$r(x_n,a_n,\mu_n) = \tilde{r}(x_n) - \frac{|a_n|}{\mathcal{|X|}} - \log(\mu_n(x_n))\;,$$
where $\tilde{r}(x_n)$ denotes the distance to the bar, whereas the last term represents the aversion of the player for crowded areas in the spirit of~\cite{MR3698446}.

\textbf{Numerical results:} We conduct an experiment with 100 states and an horizon $N=15$. Starting from a uniform distribution, we can observe in Figure~\ref{fig:FH} that both backward induction and $Q$-learning algorithms converge quickly to a peaky distribution where the representative player intends to be as close as possible to the bar while moving away if the bar is already too crowded. The exploitability offers a nice way to measure how close we are from the Nash equilibrium and shows as expected that the model-based algorithm (backward induction) converges at a rate $O(1/t)$ and faster than the model-free algorithm ($Q$-learning).

\begin{figure}[htbp]
    \centering
    \begin{subfigure}{0.33\textwidth}
      \centering
      \includegraphics[width=1.0\linewidth]{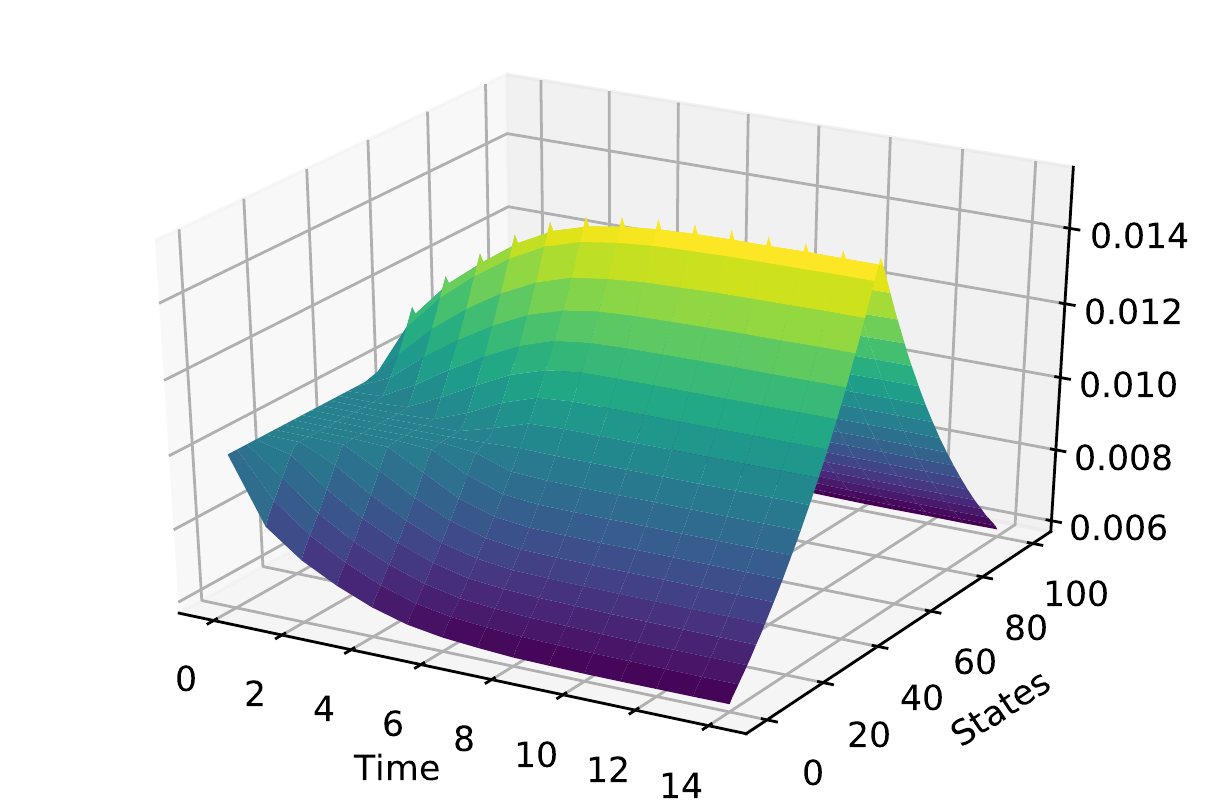}
      \caption{Model-based}
      \label{subfig:FH-BI}
    \end{subfigure}%
    \begin{subfigure}{0.33\textwidth}
      \centering
      \includegraphics[width=1.0\linewidth]{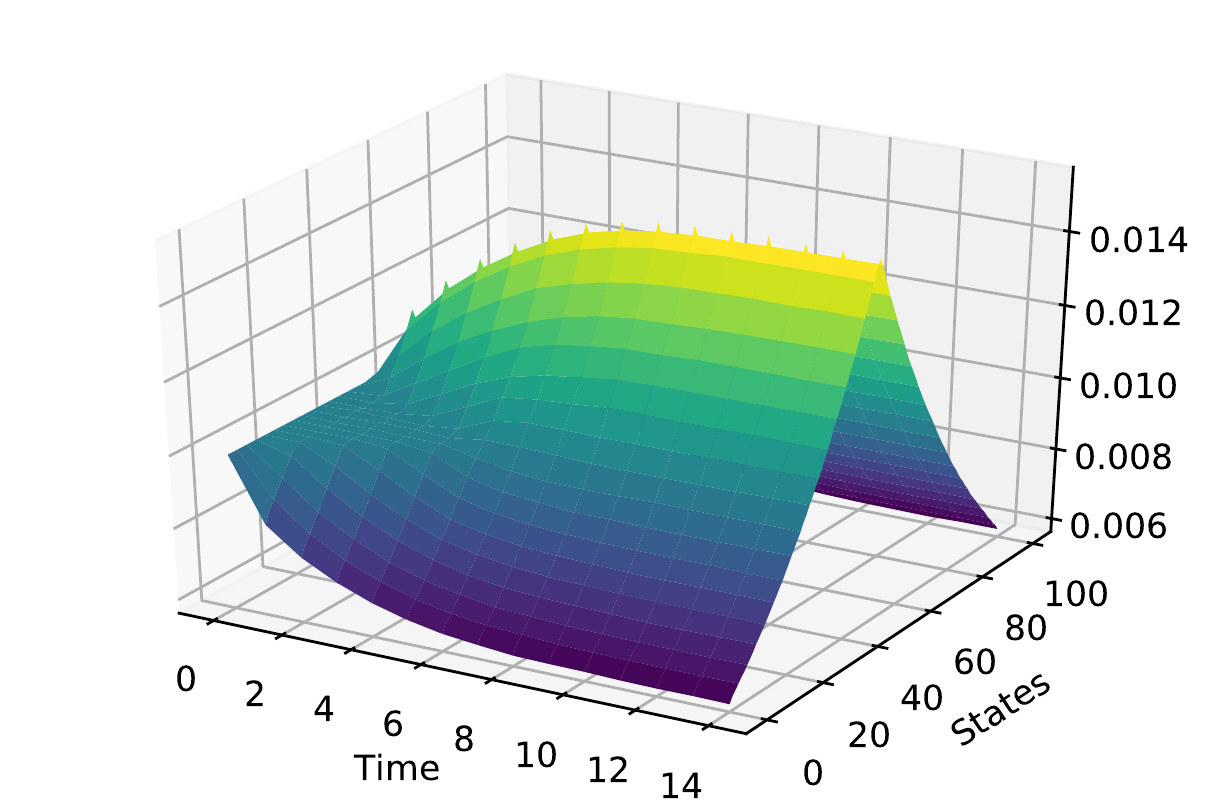}
      \caption{Model-free}
      \label{subfig:FH-Q}
    \end{subfigure}%
    \begin{subfigure}{0.33\textwidth}
      \centering
      \includegraphics[width=1.0\linewidth]{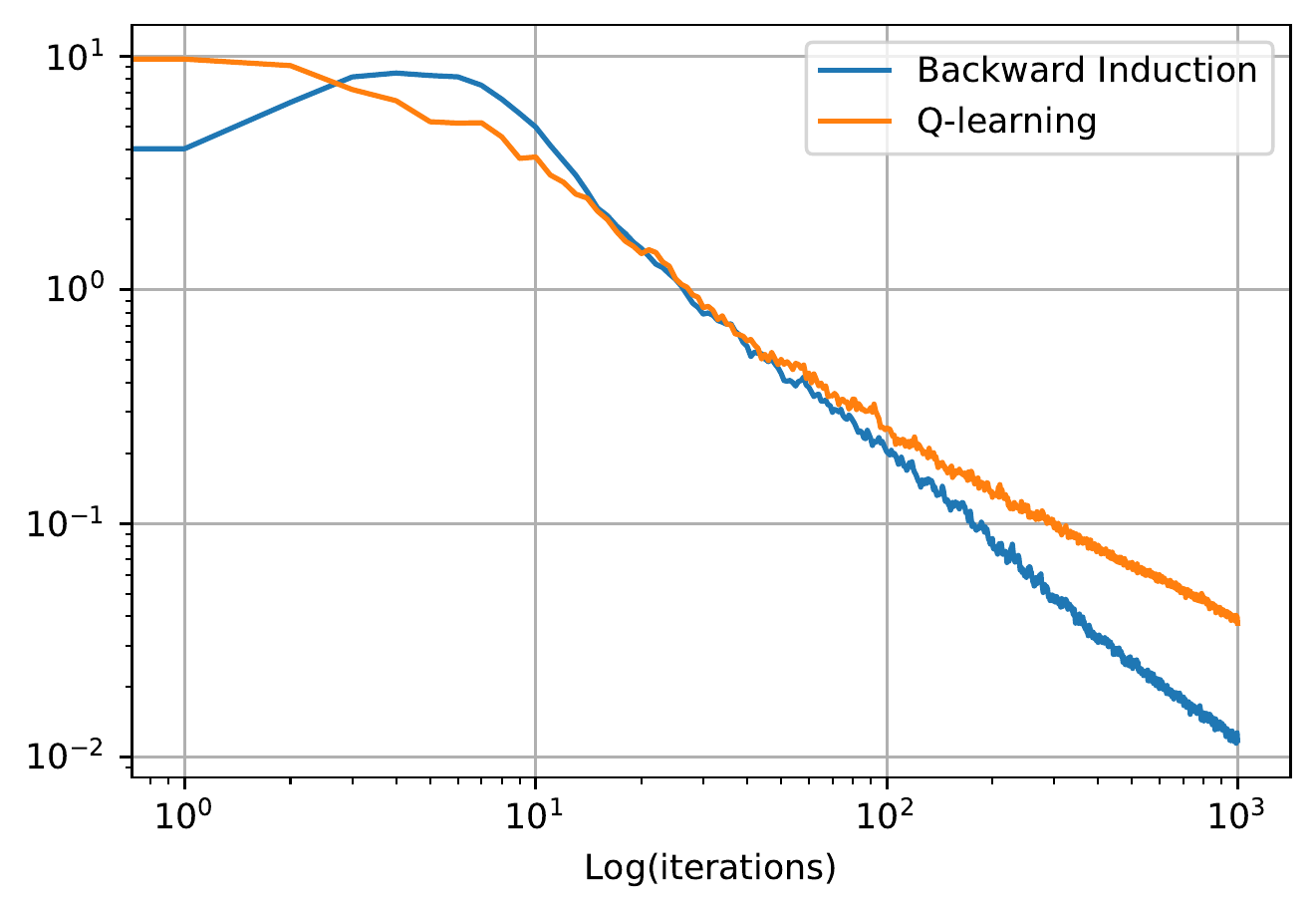}
      \caption{Exploitability}
      \label{subfig:FH-exp}
    \end{subfigure}

    \caption{Beach bar process in finite horizon: (a, b) evolution of the distribution, (c)  exploitability.
    }
    \label{fig:FH}
    \end{figure}

\section{Finite Horizon Mean Field Games with Common Noise}
We now turn to the consideration of so-called MFG with common noise, that is including an additional discrete \emph{and common} source of randomness in the dynamics. Players still sequentially take actions ($a \in \mathcal{A}$) in a state space $\mathcal{X}$, but the dynamics and the reward are affected by a common noise sequence $\{\xi_n\}_{0\leq n\leq N}$. We denote $\Xi_n = \{\xi_k\}_{0\leq k < n} = \Xi_{n-1}.\xi_{n-1} $ where $|\Xi_n|$ represents the total length of the sequence. The extra common source of randomness $\xi$ affects both the reward $r(x, a, \mu, \xi)$ and the probability transition function  $p(x'|x, a, \xi)$. We consider policies $\pi_n(a|x,\Xi)$ and population distribution $\mu_{n}(x| \Xi)$ which are both noise-dependent, and will simply be denoted $\pi_{n, \Xi}(a|x)$ and $\mu_{n| \Xi}(x)$. The $Q$ function is defined as: 
\begin{align}
    &Q^{\pi, \mu}_N(x,a|\Xi_N) = r(x, a, \mu_{N| \Xi_{N}}, \xi_N),\quad Q^{\pi, \mu}_{n-1}(x,a|\Xi_{n-1}) =\sum \limits_{\xi} P(\xi_{n-1}=\xi|\Xi_{n-1}) \Big[\\
    &\qquad r(x, a, \mu_{n-1, \Xi_{n-1}}, \xi) + \sum \limits_{x' \in \mathcal{X}}p(x'|x,a, \xi)\mathbb{E}_{b \sim \pi_{n}(.|x', \Xi_{n-1}.\xi)}\left[Q^{\pi, \mu}_{n}(x', b|\Xi_{n-1}.\xi)\right]\Big],
\end{align}
while the value function is simply 
$V_n^{\pi,\mu}(x,\Xi_{n})=\mathbb{E}_{a \sim \pi_{n,\Xi_{n}}(.|x)}\left[Q^{\pi,\mu}_n(x,a |\Xi_{n})\right]$. Similarly,
 the distribution over states is conditioned on the sequence of noises and satisfies the balance equation: $\mu^\pi_0(x, \Xi_0) = \mu_0(x)$ (with $\Xi_{0}$ being the empty sequence $\{\}$) and $\mu^\pi_{n+1}(x'|\Xi.\xi) = \sum \limits_{x \in \mathcal{X}} p^{\pi_{n,\Xi.\xi}}(x'|x,\xi)\mu^\pi_{n}(x|\Xi)$. 
The expected return for a representative player starting at $\mu_0$ is:
\begin{align}
J(\mu_0, \pi, \mu) = \sum_{x \in \mathcal{X}} \mu_0(x) V^{\pi, \mu}_0(x, \Xi_{0}) = \sum_{n=0}^N \,\, \sum_{\Xi, \xi, |\Xi|=n} P(\Xi.\xi) \sum_{x \in \mathcal{X}} \left[ \mu_{n}(x, \Xi) r(x, a, \mu_{n, \Xi}, \xi) \right]
\end{align}
with $P(\Xi_0) = 1$ and $P(\Xi.\xi) = P(\xi|\Xi) P(\Xi)$.
Finally the {\bf exploitability} is again defined as:
\begin{equation}
    \phi(\pi) = \max_{\pi'}J(\mu_0, \pi', \mu^\pi)-J(\mu_0, \pi, \mu^\pi).    
\end{equation}

\textbf{Continuous time Fictitious Play for MFGs with common noise:}
The Fictitious play process on MFGs with common noise is as follows.
For $t<1$, we start with an arbitrary policy $\bar \pi^{t<1}$ (by convention we will take $\bar \pi^{t}=\pi^{\BR,t}$ for $t<1$) whose distribution is $\bar \mu^{t<1}=\mu^{\pi^{t<1}}$ (with the convention that $\bar \mu^{t}=\mu^{\BR,t}$). Then, for all $t$ and $\Xi$:
\begin{align}
    \bar \mu_n^t (x|\Xi) = \frac{1}{t} \int \limits_{s=0}^t \mu_n^{\BR, s} (x|\Xi) ds,
\end{align}
where $\mu^{\BR, t}$ is the distribution of a best response policy $\pi^{\BR,t}$ against $\bar \mu^t$ when $t\geq 1$. The distribution $\mu^t$ is the distribution of a policy $\bar \pi^{t}$, which is defined as follows for $t\geq 1$:
\begin{align}
    \forall n, \Xi,\qquad \;\bar \pi_n^t(a|x,\Xi) \int \limits_{s=0}^t \mu^{\BR,s}_n(x|\Xi) ds = \int \limits_{s=0}^t \mu^{\BR,s}_n(x|\Xi) \pi^{\BR,s}_n(a|x,\Xi) ds.
\end{align}
\begin{theorem}
\label{thm:fp_FHCN}
Under the monotony assumption,  the exploitability is a strong Lyapunov function of the system for $t\geq 1$:
$\frac{d}{dt} \phi(\bar \pi^t) \leq - \frac{1}{t} \phi(\bar \pi^t).$
Therefore, $\phi(\bar \pi^t) = O(\frac{1}{t})$.
\end{theorem}

\section{Experiments with Common Noise}

\subsection{Linear Quadratic Mean Field Game}

\label{sec:expe-common-noise-LQ}
\textbf{Environment:} We use a similar environment as the one described in the Linear Quadratic MFG. On top of the idiosyncratic noise $\epsilon_n$, we add a common noise $\xi_n$, which is assumed to be stationary and i.i.d. We now consider the following dynamics:
\begin{align}
    x_{n+1} = x_{n} + (K(m_n - x_n) + a_n)\Delta_n + \sigma(\rho\xi_n + \sqrt{1-\rho^2}\epsilon_{n})\sqrt{\Delta_n}\;.
\end{align}
The reward remains unchanged, except that the first moment of the state distribution $\bar\mu_n$ now depends on the sequence of common noises $\Xi_n$: $m_n=\mathbb{E}[x_n | \Xi_n]$. We set $\rho=0.5$.

\begin{figure}[htbp]
    \centering
    \begin{subfigure}{0.25\textwidth}
      \centering
      \includegraphics[width=1.0\linewidth]{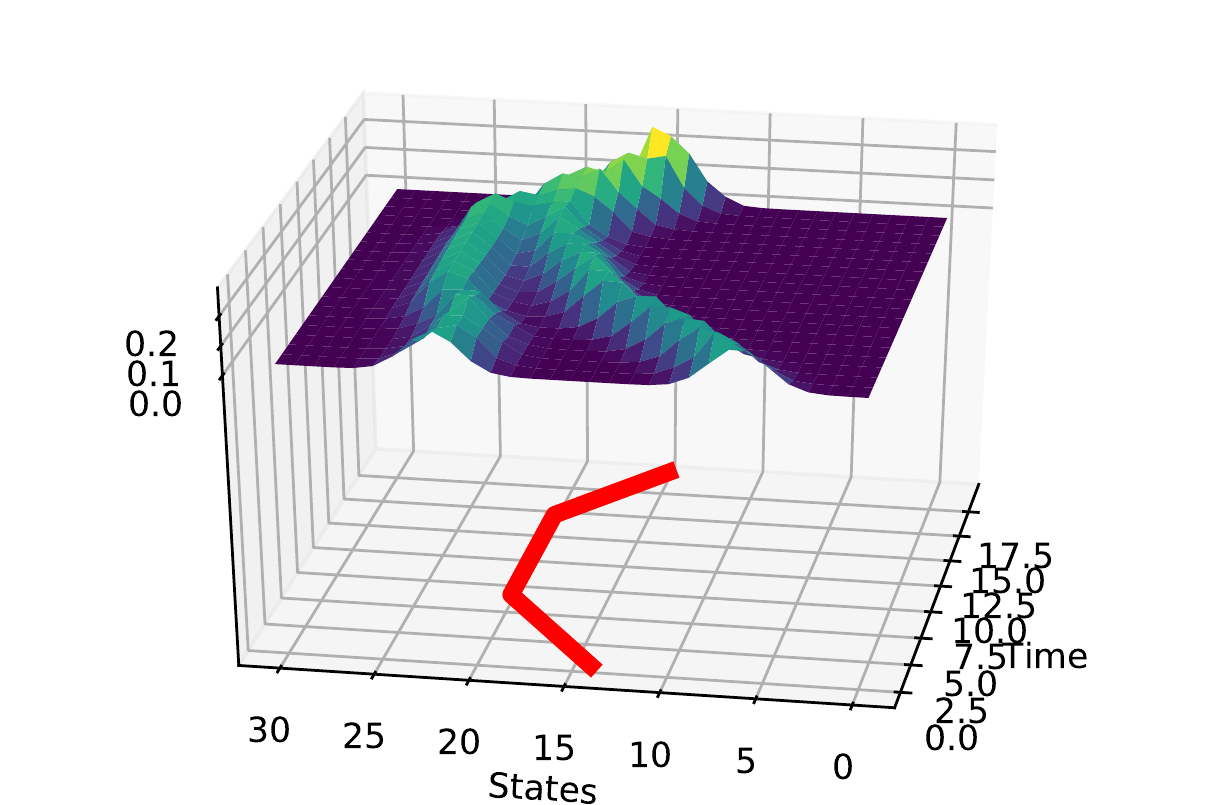}
      \caption{Exact Solution}
    \end{subfigure}%
    \begin{subfigure}{0.25\textwidth}
      \centering
      \includegraphics[width=1.0\linewidth]{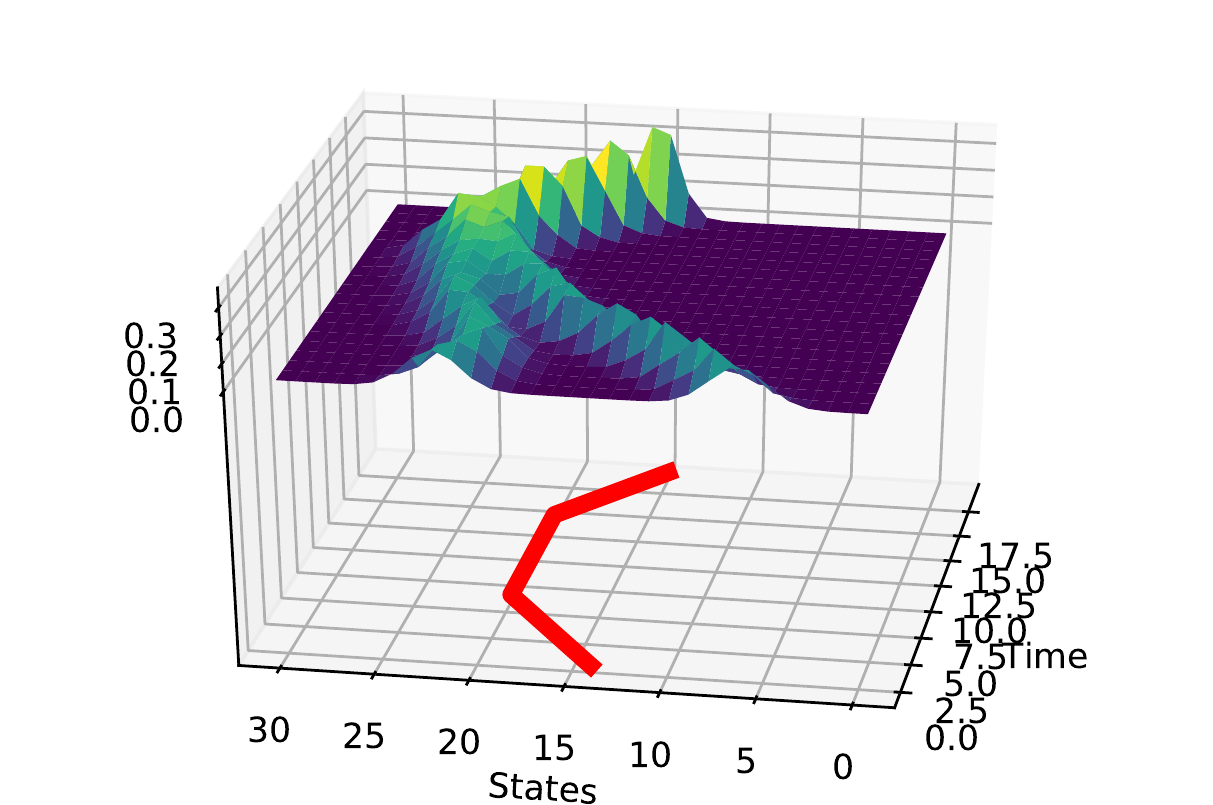}
      \caption{Model-based}
    \end{subfigure}%
    \begin{subfigure}{0.25\textwidth}
      \centering
      \includegraphics[width=1.0\linewidth]{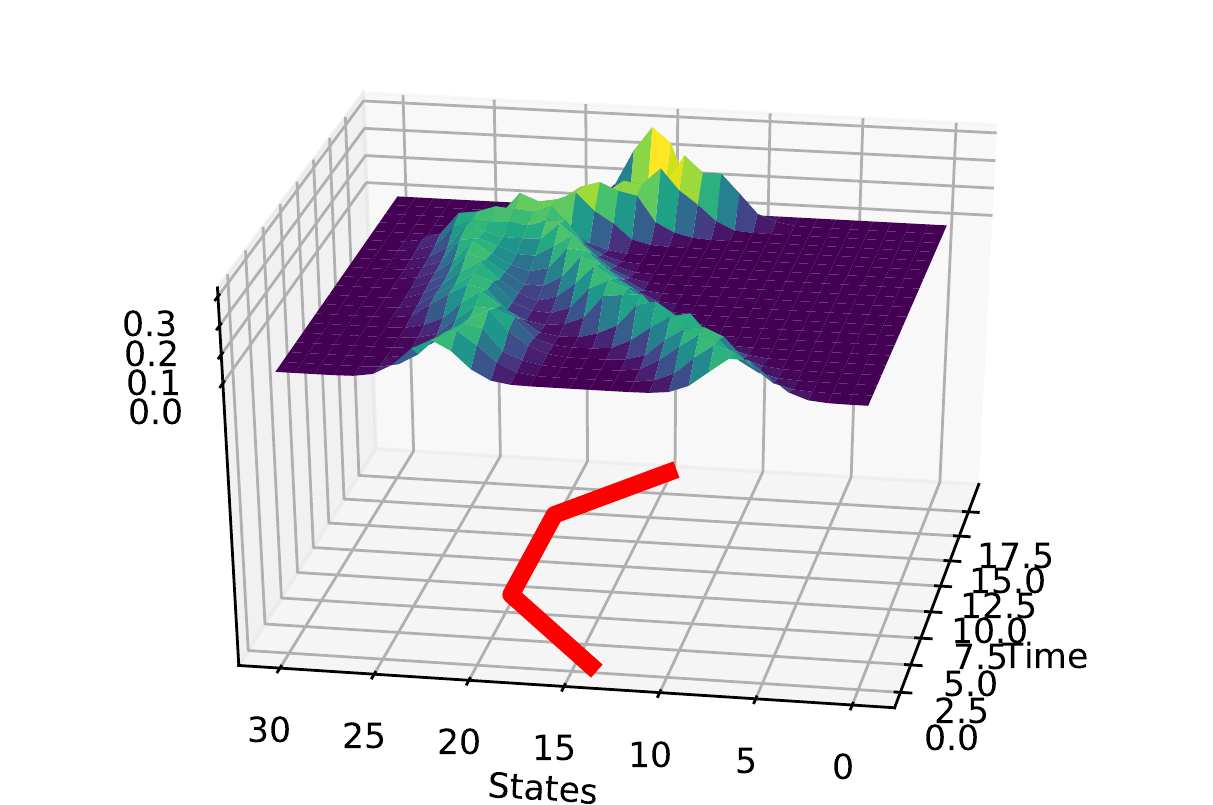}
      \caption{Model-free}
    \end{subfigure}%
    \begin{subfigure}{0.25\textwidth}
      \centering
      \includegraphics[width=1.0\linewidth]{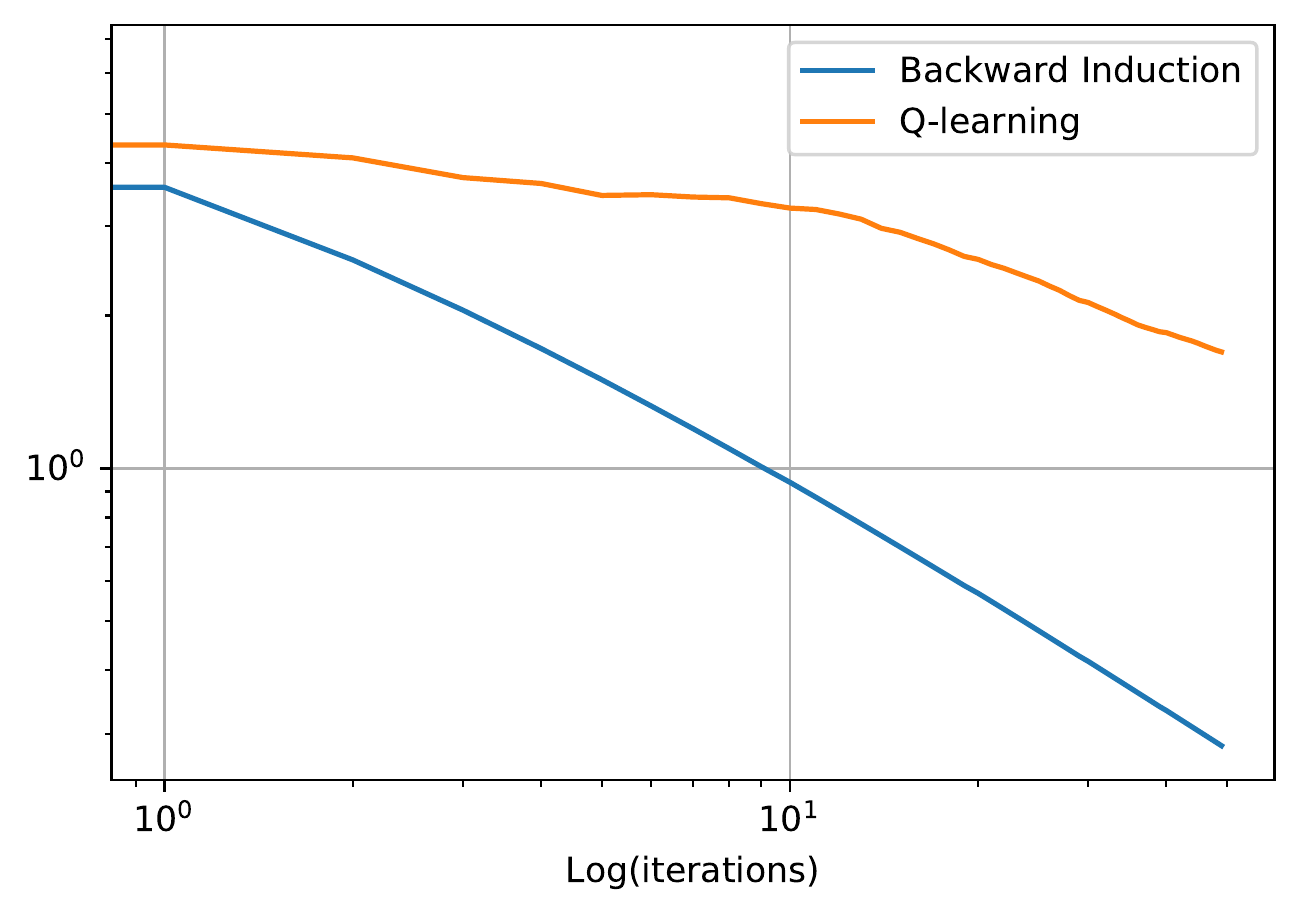}
      \caption{Exploitability}
    \end{subfigure}

    \caption{Linear Quadratic with Common Noise.}
    \label{fig:LQ_CN}
\end{figure}

\textbf{Numerical results:} On Figure \ref{fig:LQ_CN}, the two separated bell-shaped distributions reassemble and follow the sequence of common noises. Namely, the mean of the distribution moves with the successive common noises, which are represented by the red line below the distribution's evolution. This evolution can be interpreted as a school of fish which undergoes a water flow (\textit{i.e.} the sequence of common noises). Both model-based and model-free approaches approximate the exact solution. The exploitability of model-based still decreases at a rate $O(1/t)$, while the one of model-free decreases more slowly.

\subsection{The Beach Bar Process}

\textbf{Environment:} We consider a setting where the bar can close at only one given time step. This gives two possible realizations of the common noise: (1) the bar stays open or (2) it closes at this time step. Here, the dynamics remain unchanged but the reward now depends on the common noise: $r_{open}$ is the same reward as before, whereas $r_{closed}(x_n,a_n,\mu_n) = - \frac{|a_n|}{\mathcal{|X|}} - \log(\mu_n(x_n))$.

\begin{figure}[htbp]
    \centering
    \begin{subfigure}{0.33\textwidth}
      \centering
      \includegraphics[width=0.8\linewidth]{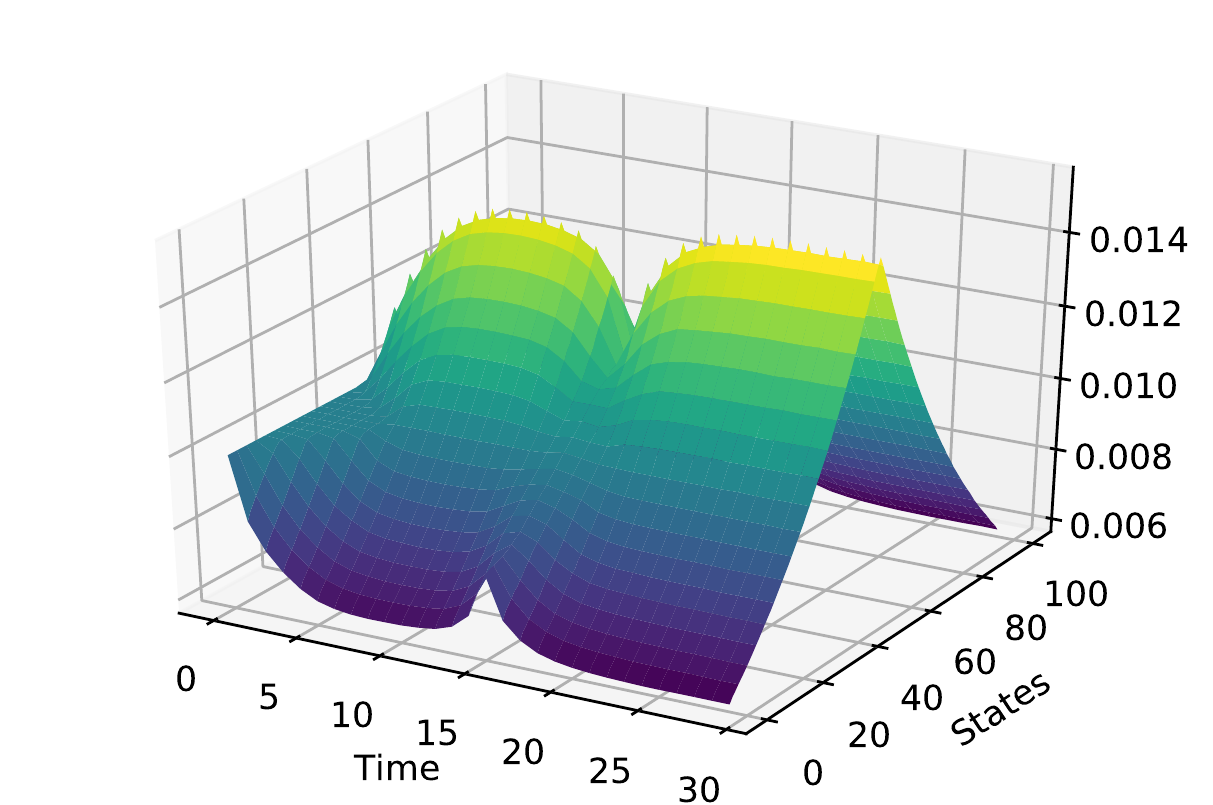}
      \captionsetup{width=.9\textwidth}
      \caption{Model-based, the bar stays open}
    \end{subfigure}%
    \begin{subfigure}{0.33\textwidth}
      \centering
      \includegraphics[width=0.8\linewidth]{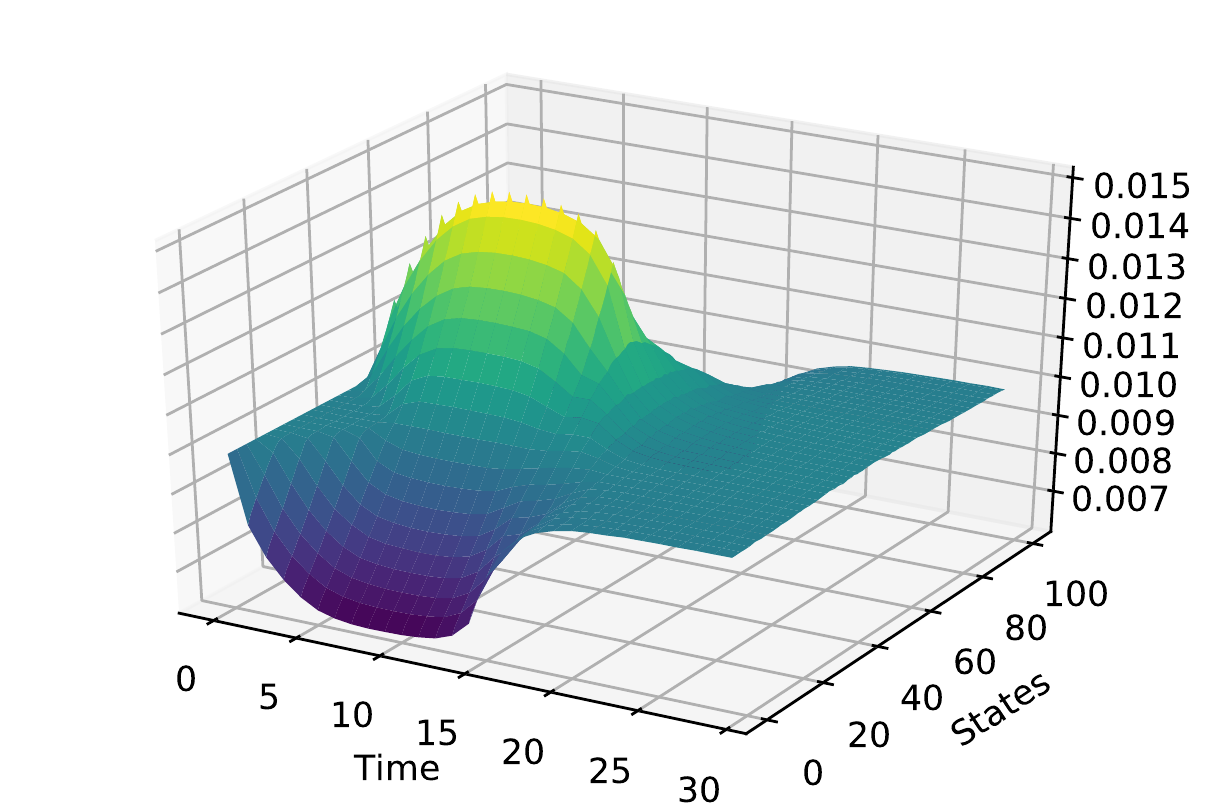}
      \captionsetup{width=.9\textwidth}
      \caption{Model-based, the bar closes}
    \end{subfigure}%
    \begin{subfigure}{0.33\textwidth}
      \centering
      \includegraphics[width=0.8\linewidth]{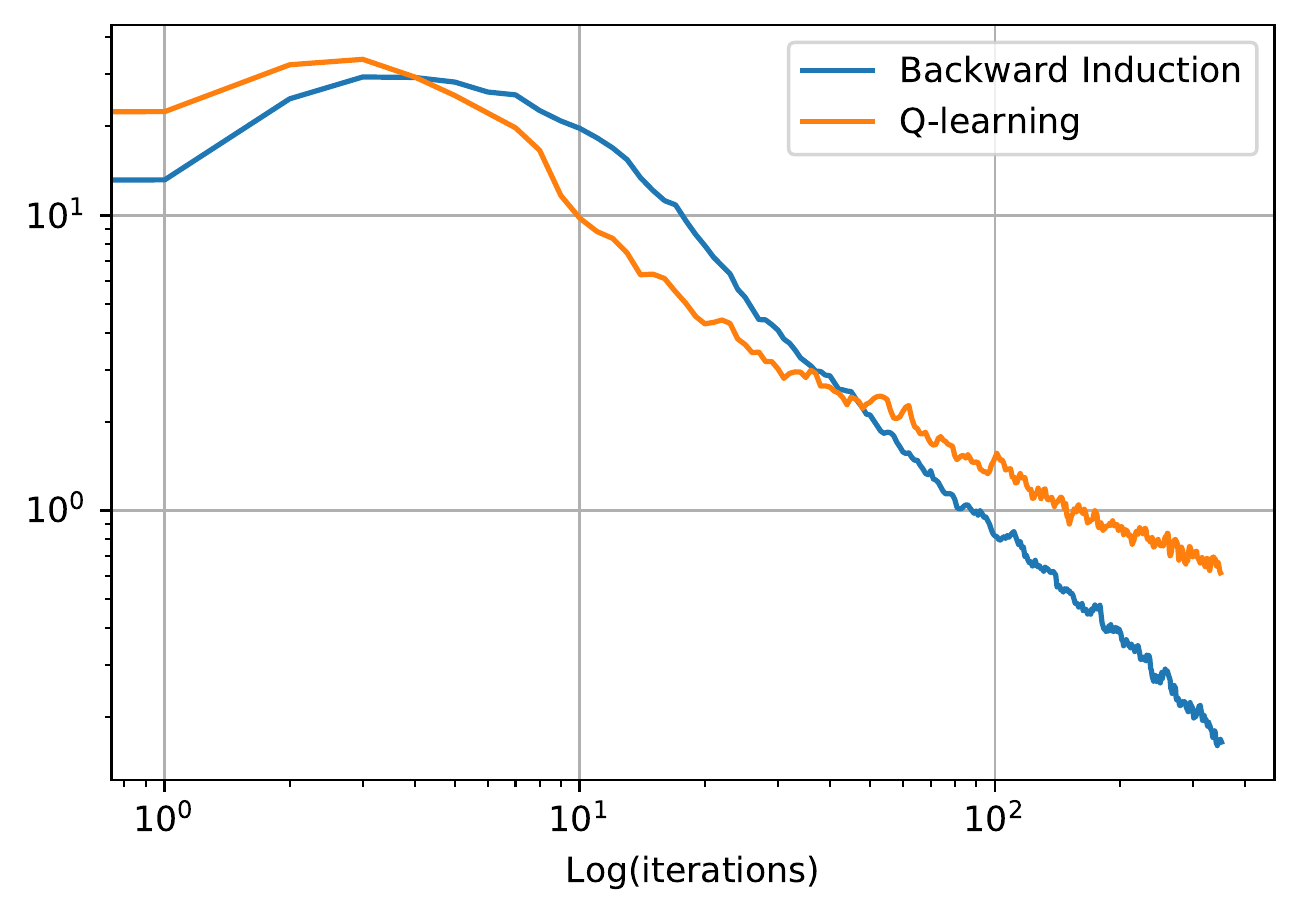}
      \caption{Exploitability}
    \end{subfigure}
    
    \medskip
    
    \begin{subfigure}{0.33\textwidth}
      \centering
      \includegraphics[width=0.8\linewidth]{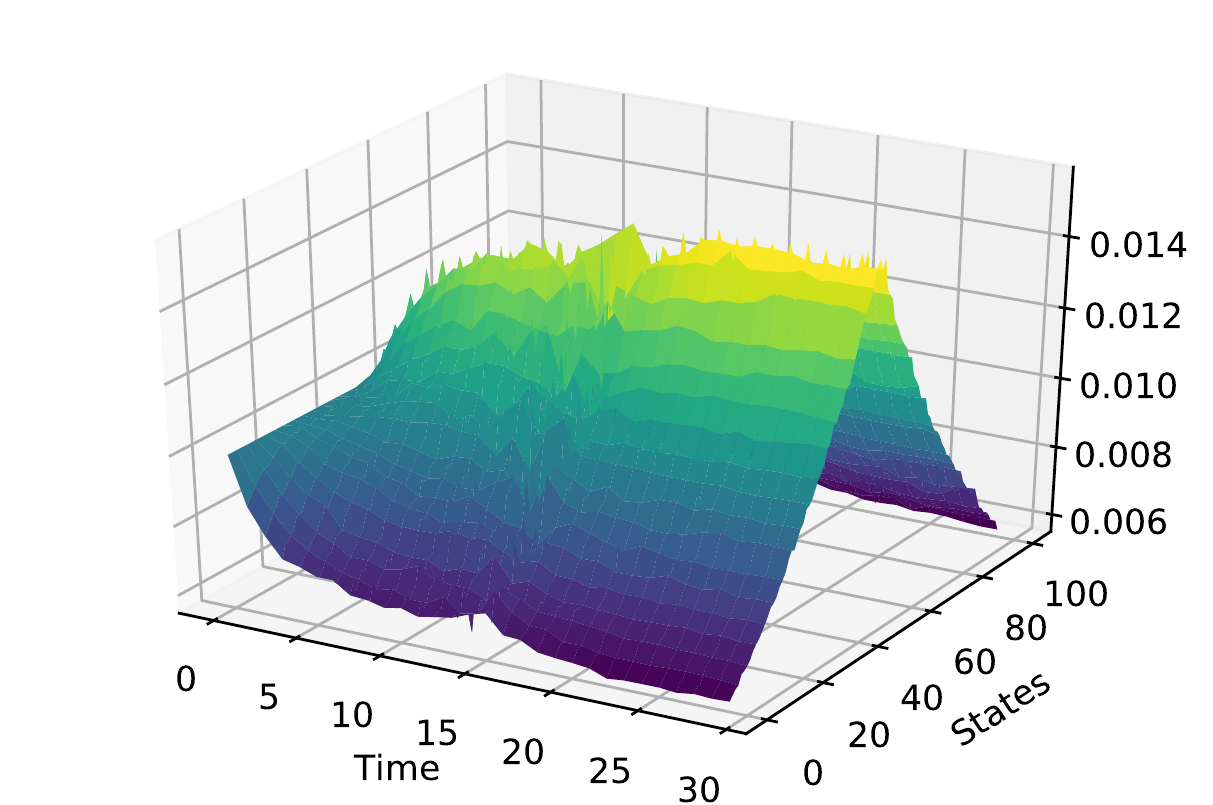}
      \caption{Model-free, the bar stays open}
    \end{subfigure}%
    \begin{subfigure}{0.33\textwidth}
      \centering
      \includegraphics[width=0.8\linewidth]{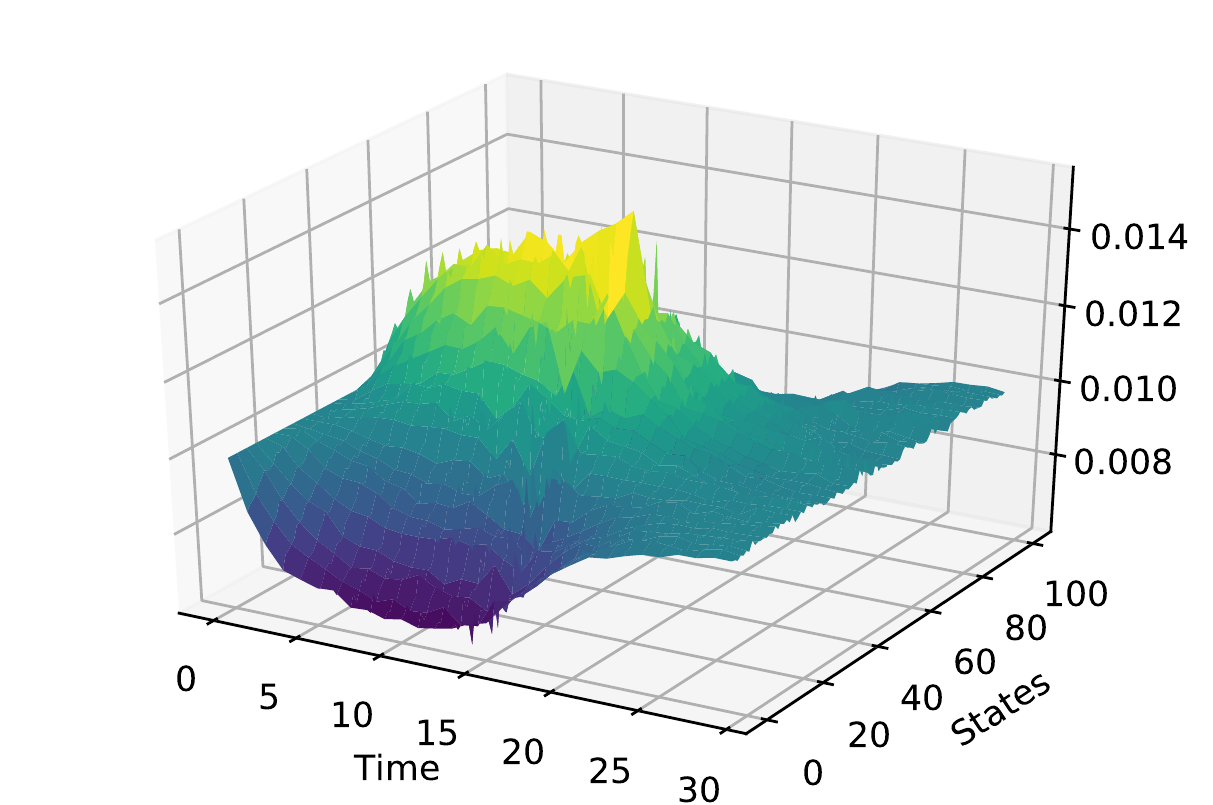}
      \caption{Model-free, the bar closes}
    \end{subfigure}%
    \caption{First Common Noise setting, the bar has a probability $0.5$ of closing at time step $15$.}
    \label{fig:1CN}
\end{figure}
\textbf{Numerical results:} We set the time step of closure at $\frac{N}{2}$ where $N=30$ is the horizon of the game and the number of states $\mathcal{|X|}$ to $100$. We choose the probability of closure to be $0.5$. Figure \ref{fig:1CN} shows that the players anticipate the possibility that the bar may close: the density of people next to the bar decreases before the time step of the common noise. After the common noise, the distribution becomes uniform if the bar has closed or people go back next to the bar if the bar stays open. Once again, the exploitability indicates that the model-based and model-free approaches both converge to the Nash equilibrium and that the model-based converges faster.

\section{Experiment at Scale}

\begin{figure}[htbp]
    \centering
    \begin{subfigure}{0.25\textwidth}
      \centering
      \includegraphics[width=1.0\linewidth]{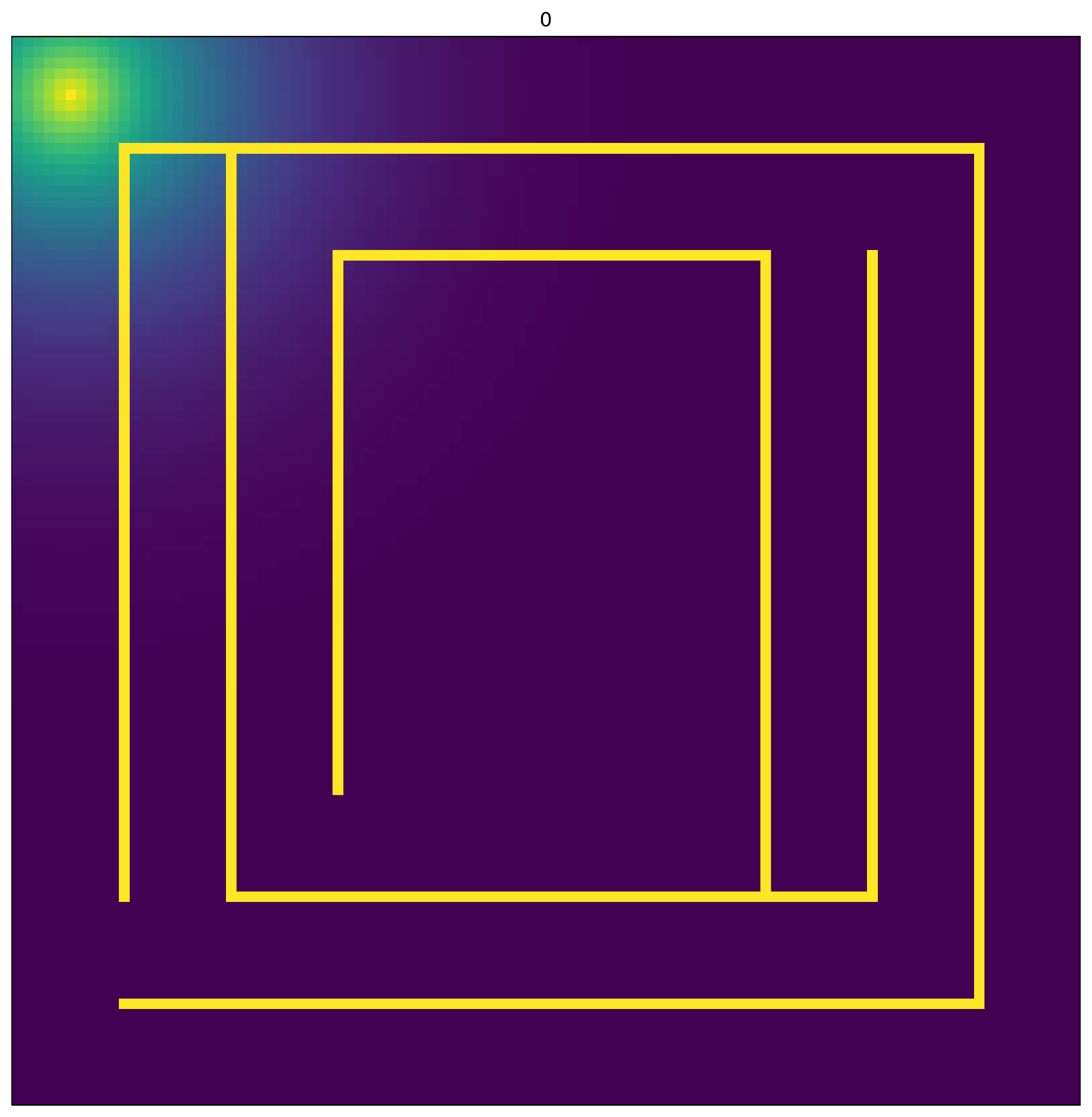}
      \caption{Start}
    \end{subfigure}%
    \begin{subfigure}{0.25\textwidth}
      \centering
      \includegraphics[width=1.0\linewidth]{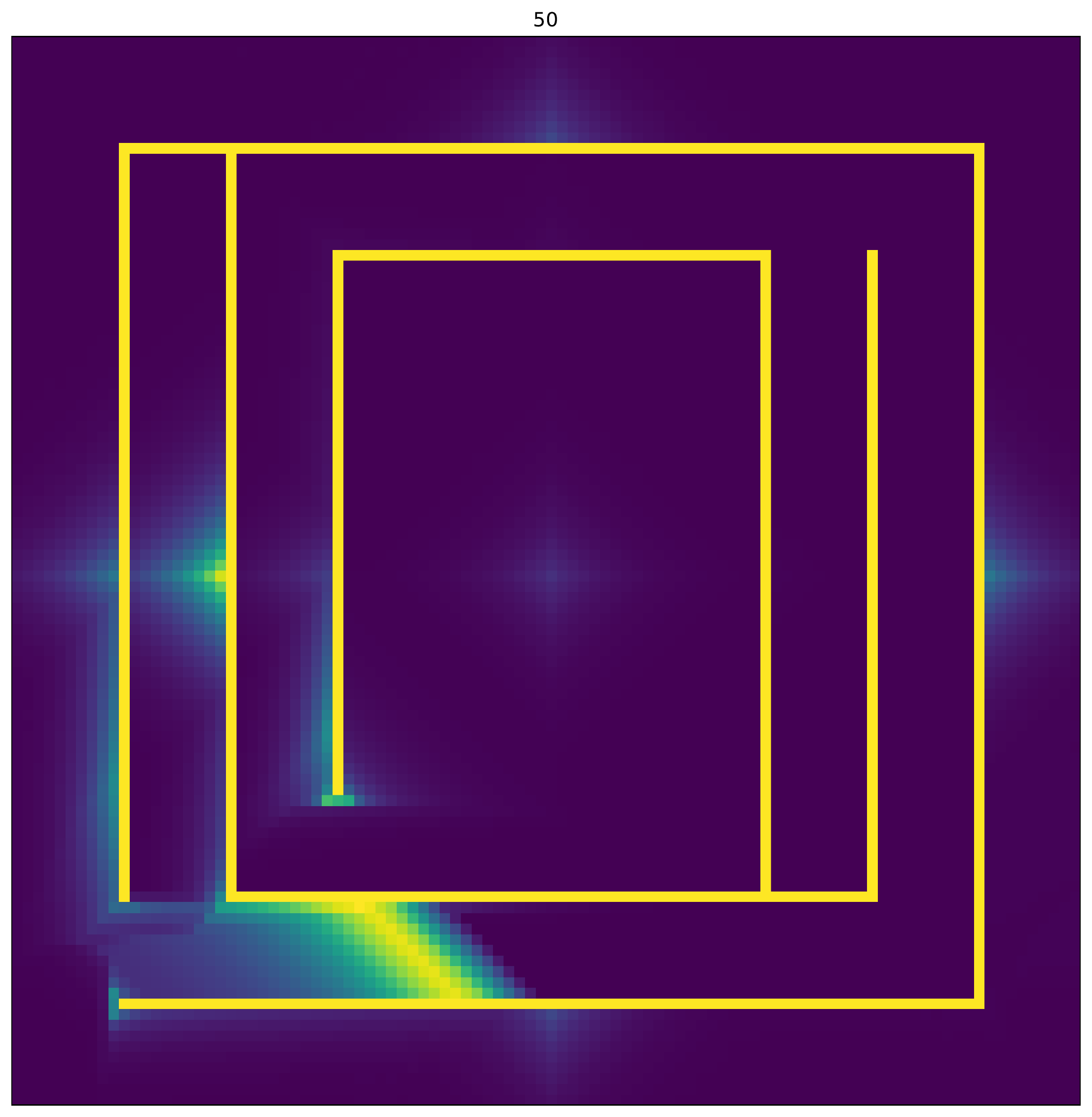}
      \caption{Middle}
    \end{subfigure}%
    \begin{subfigure}{0.25\textwidth}
      \centering
      \includegraphics[width=1.0\linewidth]{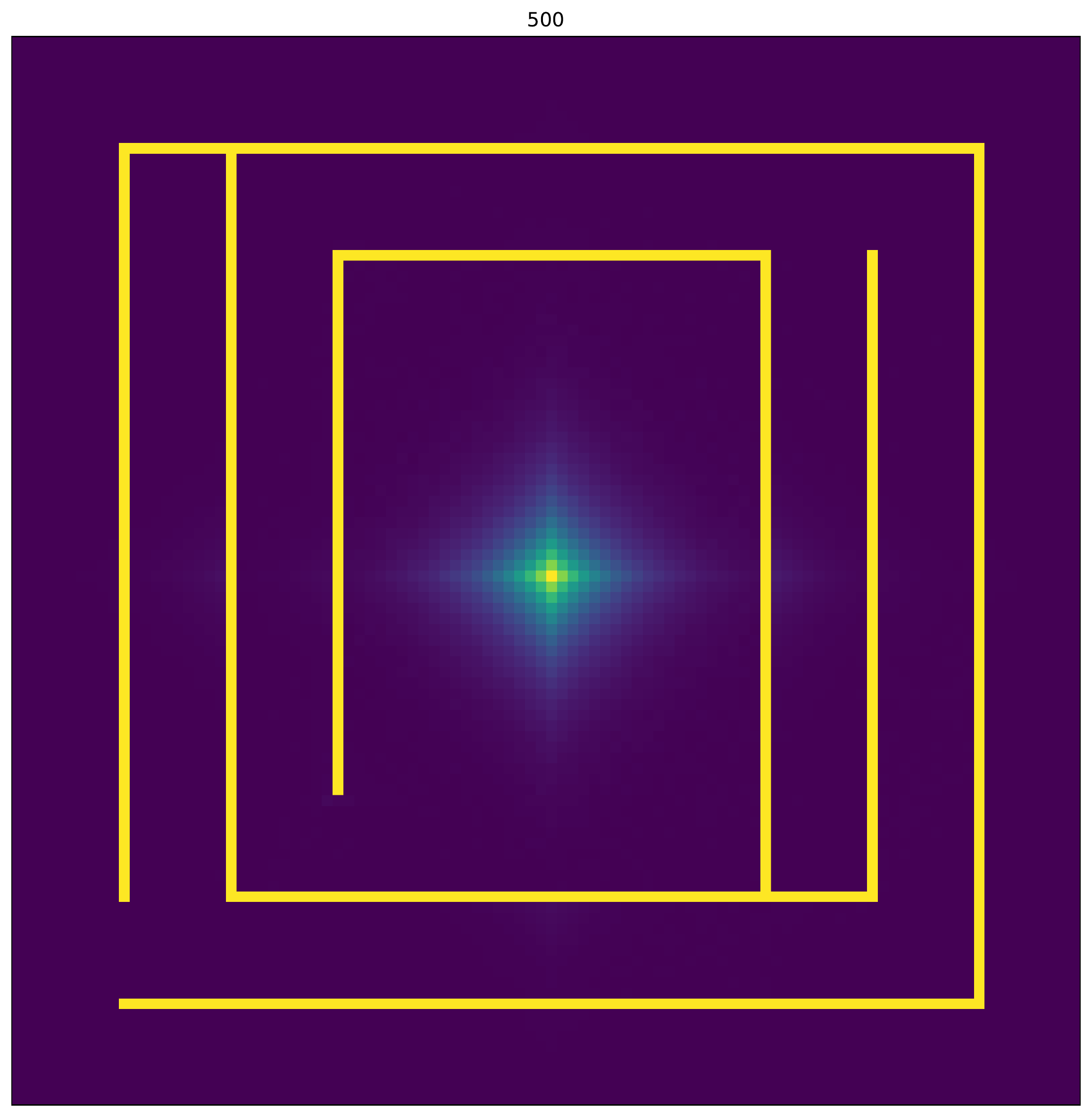}
      \caption{End}
    \end{subfigure}%
    \begin{subfigure}{0.25\textwidth}
      \centering
      \includegraphics[width=1.0\linewidth]{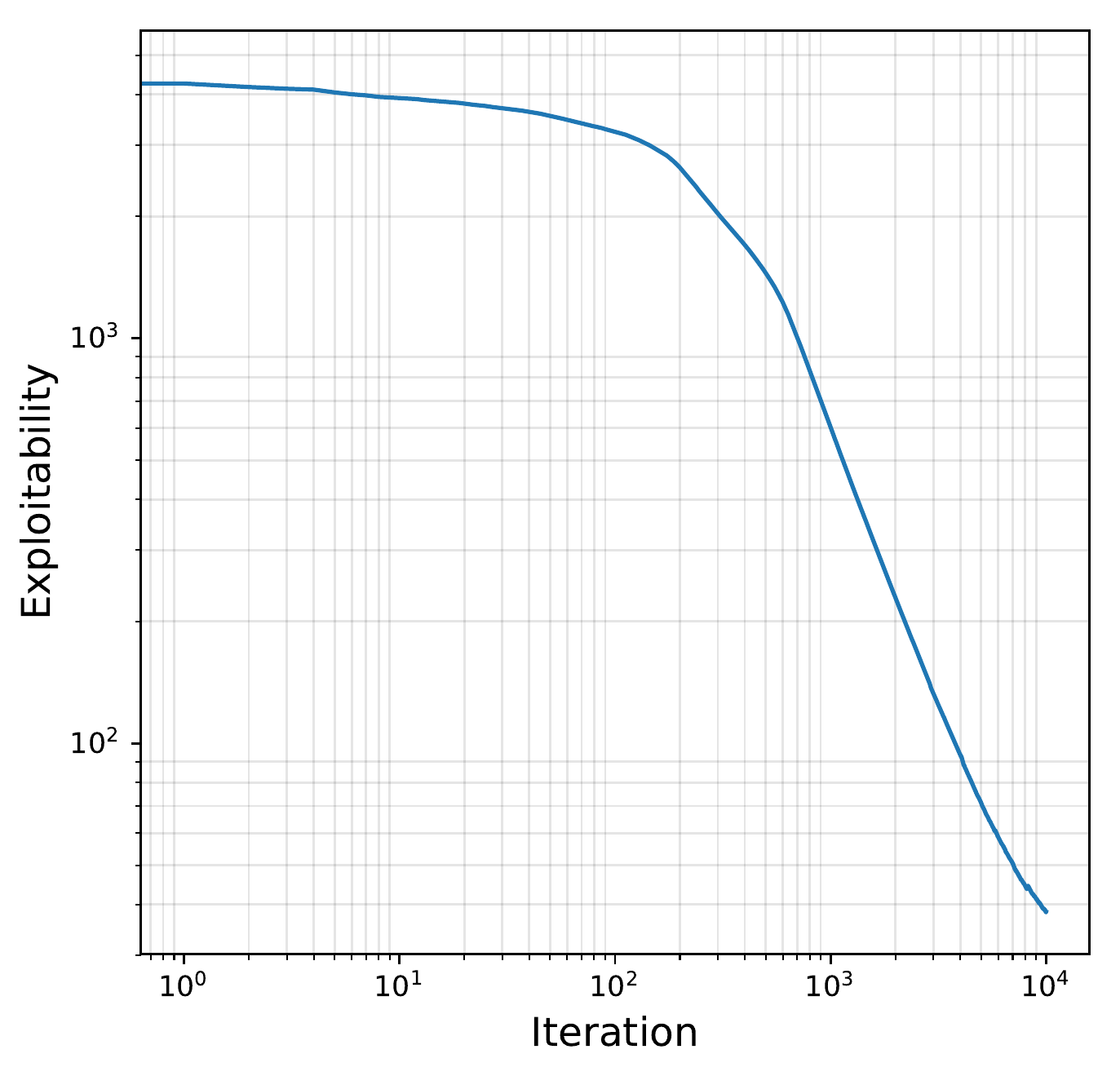}
      \caption{Exploitability}
    \end{subfigure}

    \caption{2D crowd modeling example.}
    \label{fig:2D_experiment}
\end{figure}

We finally present a crowd modeling experiment, motivated by swarm robotics (see e.g.~\cite{mcguire2019minimal,szymanski2006distributed,ducatelle2014cooperative}), where a distribution of players is encouraged to move in a maze towards the center of a $100\times 100$ grid. The reward at a state $(i, j)$ is described as $r(s=(i,j), a, \mu) = 10*(1-\frac{\|(i, j)-(50, 50)\|_1}{100}) - \frac{1}{2} \log(\mu(x))$,
where the last term captures the aversion for crowded areas. The initial distribution is chosen proportional to $(1-\frac{\|(i, j)-(5, 5)\|_2}{\sqrt{2\times 95^2}})^{10}$ while being null on the maze obstacles (the yellow strait lines). The evolution of the distribution as well as the exploitability are represented in Figure~\ref{fig:2D_experiment} (a video is available in supplementary material).

\section{Related Work}
\textbf{Theoretical results in MFGs:}
Theoretical results in terms of uniqueness, existence and stability of Nash equilibrium in such games are numerous, see  \cite{cardaliaguet2010notes,MR3134900,carmona2018probabilisticI-II}. A key motivation is that the optimal control derived in an MFG provides an approximate Nash equilibrium in a game with a large but finite number of players. In general, most games are considered in a continuous setting while Gomes \textit{et al.} \cite{gomes2010discrete} proved existence results for finite state and action spaces MFGs and~\cite{MR3880245} considered finite state discounted cost MFGs. An important and challenging extension is the case of players sharing a common source of risk (such as several companies in the same economy market), giving rise to the so-called MFG with common noise, see \cite{carmona2016mean} or  \cite[Volume II]{carmona2018probabilisticI-II}. These games are usually solved by numerical methods for partial differential equations~\cite{achdoulauriere2020mfgnumerical} or probabilistic methods \cite{angiuli2019cemracs,CarmonaLauriere_DL,fouque2019deep}.

\textbf{Learning in games and MFGs:} The scaling limitations of traditional multi-agent learning methods with respect to the number of players remain quite hard to overcome as the complexity of independent learning methods~\cite{foerster2018learning, perolat2018actor, perolat2020poincar, srinivasan2018actor, omidshafiei2019neural, foerster2018counterfactual, foerster2017stabilising} scales at least linearly with the number of players and some methods may scale exponentially (\textit{e.g.} Nash $Q$-learning~\cite{hu2003nash} or correlated $Q$-learning~\cite{greenwald2003correlated}). By approximating the discrete population by a continuous one, the MFG scheme made learning approaches more suitable and attracted a surge of interest.
Model-based methods have been first considered (\textit{e.g.} ~\cite{yin2010learning} studied a MF oscillator game,~\cite{cardaliaguet2017learning} initiated the study of Fictitious Play in MFGs). Recently, several works have focused on model-free methods such as $Q$-learning~\cite{guo2019learning} but the convergence results rely on very strong hypotheses. Note that, although our method can make use of $Q$-learning to learn a best response, it does not rely on it. Also, our method can make use of both model-based and model-free algorithms. Finally, our method relies only on the Lasry-Lions monotonicity condition, which is much less restrictive than a potential or variational structure.
 
Fictitious Play (FP), which is also a classical method to learn in $N$-player games~\cite{robinson1951iterative, ostrovski2013payoff, harris1998rate, hofbauer2002global, heinrich2015fictitious, perolat2018actor}, combined with a model-free algorithm has been considered in \cite{mguni2018decentralisedli} but with several inaccuracies, as already pointed out in~\cite{subramanianpolicy}, which focuses on policy gradient methods. 
However, they study a restricted stationary setting as opposed to the finite time horizon covered by our contribution and their convergence results hold under hardly verifiable assumptions. 

Convergence of approximate FP has been proved in~\cite{elie2020convergence} (based on the FP analysis of~\cite{hadikhanloo2019finite}) but without common noise and their analysis is for discrete time FP and only for first-order MFGs (without noise in the dynamics). Our analysis, done in continuous time, is more transparent and works for MFGs with both idiosyncratic and common sources of randomness in the dynamics. Furthermore, their numerical example was stationary whereas we were also able to learn the solution of time-dependent MFGs, which covers a larger scope of meaningful applications. Finally, our analysis provides a rate of convergence ($O(\frac{1}{t})$) while previous FP work in MFG do not.
\section{Conclusion}
In this paper we have shown that Fictitious Play can serve as a basis for building practical algorithms to solve a wide variety of MFGs including finite horizon and $\gamma$-discounted MFGs as well as games perturbed by a common noise. We proved that, in all these settings, the resulting exploitability decreases at a rate of $O(\frac{1}{t})$ and that this metrics can be used to monitor the quality of the control throughout the learning. To illustrate our findings and the versatility of the method, we  instantiated the Fictitious Play scheme using Backward Induction and $Q$-Learning to learn intermediate best responses. Application of these instances on different MFGs have shown that the proposed algorithms consistently learned a near-optimal control and led to the desired behaviour for the population of players. This scheme has the potential to scale up dramatically by using advanced reinforcement learning algorithms combined with neural networks for the computation of the best response.

\newpage

\section*{Broader Impact}

    \textbf{Applications of MFGs:} The MFG model has inspired numerous applications \cite{gueant2011mean} and we hope our work can help practitioners to solve MFGs problems at scale. A popular application focuses on population dynamics modeling~\cite{achdou2017mean,cardaliaguet2016segregation} including crowd motion modeling \cite{achdou2019mean, burger2013mean,djehiche2017mean,aurell2019modeling,achdou2016mean,chevalier2015micro}, opinion dynamics and consensus formation \cite{stella2013opinion,bauso2016opinion,parise2015network}, autonomous vehicles \cite{huang2019game,shiri2019massive} or sanitary vaccination \cite{hubert2018nash,elie2020contact}.
But MFGs have also naturally found applications in banking, finance and economics including banking systemic risk \cite{MR3325083,elie2020large}, high frequency trading \cite{lachapelle2016efficiency,cardaliaguet2018mean},
    income and wealth distribution \cite{achdou2017income}, economic contract design \cite{elie2019tale}, economics in general \cite{achdou2017income,achdou2014pde,chan2015bertrand,gomes2014socio, djehiche2016mean} or price formation \cite{MR2295621,lachapelle2016efficiency,gomes2020mean}. Energy management or production applications are studied in \cite{alasseur2020extended,couillet2012electrical,elie2019mean, bagagiolo2014mean,kizilkale2019integral,li2016mean, gueant2011mean,achdou2016long,chan2017fracking,graber2018existence}, whereas security and communication applications appear in \cite{meriaux2012mean,samarakoon2015energy,hamidouche2016mean,yang2017mean, kolokoltsov2016mean,kolokoltsov2018corruption}.

\textbf{Exploitability as a metric:} One of the leading factor of progress for numerical or learning methods is the clear understanding of which metrics should be optimized. In reinforcement learning, the mean human normalized score is a standard metric of success. In supervised learning, the top 1 accuracy has been the foremost metric of success. We hope the exploitability can achieve such a role on the numerical aspects of MFGs.

\bibliographystyle{plain}
\bibliography{biblio.bib}

\newpage
\appendix
\section{Continuous Time Fictitious Play in Finite Horizon}
\label{app:CFP_FH}
In this section, we prove the Fictitious Play convergence result in the absence of common noise. For the sake of clarity, we will write:
$$r^\pi(x, \mu) = \mathbb{E}_{a \sim \pi(.|x)}\left[r(x, a, \mu)\right] \quad \mbox{ and }  \quad p^\pi(x'|x) = \mathbb{E}_{a \sim \pi(.|x)}\left[p(x'|x,a)\right]$$ for the rest of this section.

First, we prove the following property, which stems from monotonicity.

\begin{property}
\label{prop:mono-r-dmu}
Let $f$ be a smooth enough function and let assume that the ODE $\dot \mu = f(\mu)$ (with $\dot \mu = \frac{d}{dt} \mu$) has a solution $(\mu^t)_{t \ge 0} = (\mu^t_n(x))_{t \ge 0, x \in \mathcal{X}}$.
If the game is monotone, then: 
$$\sum \limits_{x \in \mathcal{X}}<\nabla_{\mu} \bar r(x,\mu), \dot \mu> \dot \mu(x) \leq 0.
$$
\end{property}
\begin{proof}
The monotonicity condition implies that, for all $\tau \geq 0$, we have:
$$\sum \limits_{x \in \mathcal{X}} (\mu^t(x) - \mu^{t+\tau}(x))(\bar r(x,\mu^t) - \bar r(x,\mu^{t+\tau}))\leq 0.
$$
Thus:
$$\sum \limits_{x \in \mathcal{X}} \frac{\mu^t(x) - \mu^{t+\tau}(x)}{\tau}\frac{\bar r(x,\mu^t) - \bar r(x,\mu^{t+\tau})}{\tau}\leq 0.
$$
The result follows when $\tau \rightarrow 0$.
\end{proof}

\begin{property}
\label{prop:avg-policy}
Let $\hat\pi^t = (\hat\pi^t_n)_{n=0,\dots,N}$ be a sequence of time-dependent policies and let $\mu^{\hat\pi^t} = (\mu^{\hat\pi^t}_n(x))_{n=0,\dots,N, x \in \mathcal{X}}$ be the sequence of their distributions over states.
Let us denote, for all $t,n,x$, $\bar \mu_n^t(x) = \frac{1}{t} \int \limits_{0}^t \mu_n^{\hat\pi^s}(x) ds$. Then, the policy generating this average distribution is:
\begin{equation}
\label{eq:bar-pi-nt}
    \bar \pi_n^t(a|x) = \frac{\int \limits_{0}^t \mu_n^{\hat\pi^s}(x) \hat\pi_n^s(a|x) ds}{\int \limits_{0}^t \mu_n^{\hat\pi^s}(x) ds}.
\end{equation}
Note that $\int \limits_{0}^t \mu_n^{\hat\pi^s}(x) ds$ can be chosen to be strictly positive as one can choose an arbitrary policy on the time interval $[0, 1]$ (for example, the uniform policy).

Or, more simply, one can write:
\begin{equation}
    \label{eq:bar-pi-nt-integral}
    \bar \mu_n^{t}(x) \bar \pi_n^t(a|x) = \frac{1}{t} \int \limits_{0}^t \mu_n^{\hat\pi^s}(x) \hat\pi_n^s(a|x) ds.
\end{equation}
Moreover, we have:
\begin{align}
    \dot {\bar \mu}_n^t(x) \bar \pi_n^t(a|x) + \bar \mu_n^t(x) \dot {\bar \pi}_n^t(a|x) = \frac{1}{t}\left[ \mu_n^{\hat\pi^t}(x) \hat\pi^t(a|x)-\bar \mu_n^{\hat\pi^t}(x) \bar \pi_n^t(a|x)\right]. 
\end{align}
\end{property}

\begin{proof}
Let us start with the following equality, which holds by definition of the dynamics:

\begin{align}
    &\mu_{n+1}^{\hat\pi^s} (x') = \sum \limits_{x \in \mathcal{X}} \sum \limits_{a \in \mathcal{A}} p(x'|x, a)\hat\pi_n^s(a|x) \mu_{n}^{\hat\pi^s} (x).
\end{align}
Then, taking on both sides the average over the Fictitious Play time yields:
\begin{align}
    &\frac{1}{t}\int \limits_{0}^{t} \mu_{n+1}^{\hat\pi^s} (x') ds 
    = \sum \limits_{x \in \mathcal{X}} \sum \limits_{a \in \mathcal{A}} p(x'|x, a) \frac{1}{t}\int \limits_{0}^{t} \hat\pi_n^s(a|x) \mu_{n}^{\hat\pi^s} (x) ds.
\end{align}

The left hand side is $\bar \mu_{n+1}^{t}(x')$ by definition, and the time average in the right hand side can be written as:
$$
    \frac{1}{t}\int \limits_{0}^{t} \hat\pi_n^s(a|x) \mu_{n}^{\hat\pi^s} (x) ds
    =
    \frac{\int \limits_{0}^{t} \hat\pi_n^s(a|x) \mu_{n}^{\hat\pi^s} (x) ds}{\int \limits_{0}^{t} \mu_n^{\hat\pi^s} (x) ds} \frac{1}{t}\int \limits_{0}^{t} \mu_n^{\hat\pi^s} (x) ds
    =
    \bar \mu_n^{\hat\pi^t}(x) \bar \pi_n^t(a|x).
$$
Combining the terms, we obtain:
$$
    \bar \mu_{n+1}^{t}(x') 
    = 
    \sum \limits_{x \in \mathcal{X}} \sum \limits_{a \in \mathcal{A}} p(x'|x, a) \bar \mu_n^{\hat\pi^t}(x) \bar \pi_n^t(a|x),
$$
which proves that the policy $\bar \pi_n^t$ defined in~\eqref{eq:bar-pi-nt} indeed generates $\bar \mu_n^t$. The other equalities in the statement can be deduced from here readily.
\end{proof}

Based on the above properties, we now proceed to the proof of the convergence of Fictitious Play (Theorem~\ref{thm:fp_FH}) in the finite horizon case.
\begin{proof}[Proof of Theorem~\ref{thm:fp_FH}]
To alleviate the notation, given a policy $\pi$, we denote:
$$
    r^\pi(x,\mu) = \sum_{a} \pi(a|x) r(x,a,\mu).
$$
We start by noticing that, thanks to the structure of the reward coming from the monotonicity assumption,
\begin{equation}
    \label{eq:d-r-tmp}
    \nabla_{\mu}r^{\pi_n^{\BR, t}}(x, \bar \mu_n^t) = \nabla_{\mu}\bar r(x,\bar  \mu_n^t) \textrm{ and } \nabla_{\mu}r^{\pi_n}(x,\bar \mu_n^t) = \nabla_{\mu}\bar r(x, \bar \mu_n^t).
\end{equation}
Moreover, from Property~\ref{prop:avg-policy} with $\pi$ replaced by $\hat\pi^{\BR}$ and $\mu^{\hat\pi^t}$ replaced by $\mu^{\BR,t}$, we obtain~\eqref{eq:bar-pi-nt-integral}. Dropping the overlines to alleviate the presentation (so $\mu^t$ and $\pi^t$ denote respectively the average sequence of distributions and the average sequence of policies), it implies:
\begin{equation}
    \label{eq:d-pi-tmp}
    \mu_n^t(x) \frac{d}{dt} \pi_n^t(a|x)  = \frac{1}{t} \mu^{\BR,t}_n(x) [\pi^{\BR,t}_n(a|x) - \pi_n^t(a|x)].
\end{equation}

Moreover, recall that:
\begin{equation}
    \label{eq:d-mu-tmp}
    \frac{d}{dt} \mu_n^t(x)
    = \frac{1}{t}\left[\mu^{\BR, t}_n(x) -\mu_n^t(x)\right].
\end{equation}
From the above observations, we deduce successively:
\begin{align*}
    \frac{d}{dt} \phi(\pi^t) 
    &= \frac{d}{dt}\left[\max \limits_{\pi'} J(\mu_0, \pi', \mu^{t}) - J(\mu_0, \pi^t, \mu^{t})\right]
    \\
    &=\sum \limits_{n=0}^{N} \sum \limits_{x \in \mathcal{X}} \Big[<\nabla_{\mu}r^{\pi_n^{\BR}}(x, \mu_n^t), \frac{d}{dt} \mu_n^t>\mu^{\BR, t}_n(x) - <\nabla_{\mu}r^{\pi_n}(x, \mu_n^t), \frac{d}{dt} \mu_n^t>\mu_n^t(x)
    \\
    &\qquad\qquad\qquad\qquad - <\frac{d}{dt} \pi_n^t(.|x), r(x,. ,\mu_n^t)>\mu_n^t(x) - r^{\pi_n}(x, \mu_n^t) \frac{d}{dt} \mu_n^t(x)\Big]
    \\
    &= \sum \limits_{n=0}^{N} \sum \limits_{x \in \mathcal{X}} \left[t <\nabla_{\mu}\bar r(x, \mu_n^t), \frac{d}{dt} \mu_n^t> \frac{1}{t}\left(\mu^{\BR, t}_n(x)-\mu_n^t(x) \right)\right]
    \\
    &\qquad+ \sum \limits_{n=0}^{N} \sum \limits_{x \in \mathcal{X}}  \left[\frac{1}{t}r^{\pi_n}(x, \mu_n^t)\mu_n^t(x) - \frac{1}{t}r^{\pi_n^{\BR,t}}(x,\mu_n^t)\mu_n^{\BR, t}(x)\right]
    \\
    &=  - \frac{1}{t} \phi(\pi^t) + \sum \limits_{n=0}^{N} \sum \limits_{x \in \mathcal{X}} \left[t <\nabla_{\mu}\bar r(x, \mu_n^t), \frac{d}{dt} \mu_n^t> \frac{d}{dt} \mu_n^t(x)\right],
\end{align*}
where the third equality holds by~\eqref{eq:d-r-tmp}, ~\eqref{eq:d-pi-tmp} and~\eqref{eq:d-mu-tmp}. Note that the product $<\nabla_{\mu}\bar r(x, \mu_n^t), \frac{d}{dt} \mu_n^t>$ in the last sum above is non-positive thanks to Property~\ref{prop:mono-r-dmu} (\textit{i.e.}, thanks to the monotonicity assumption). Hence, the conclusion holds.
\end{proof}

\section{Continuous Time Fictitious Play in Finite Horizon with Common Noise}
In this section, we prove the convergence result of continuous time Fictitious Play in finite horizon MFGs with common noise (Theorem~\ref{thm:fp_FHCN}). The reasoning is similar as in the finite horizon case without common noise (Appx.~\ref{app:CFP_FH}). The only difference comes from the conditioning with the common noise.

\begin{proof}[Proof of Theorem~\ref{thm:fp_FHCN}]

For any policy, recall that we write $\pi_{n, \Xi}^t(a|x) = \pi_{n}^t(a|x, \Xi)$. 

We first note that, by the structure of the reward function, we have,
$$
    \nabla_{\mu}r^{\pi_{n, \Xi.\xi}^{\BR, t}}(x, \mu_{n|\Xi}^t) = \nabla_{\mu}\bar r(x, \mu_{n|\Xi}^t) \textrm{ and } \nabla_{\mu}r^{\pi_{n, \Xi.\xi}}(x, \mu_{n|\Xi}^t) = \nabla_{\mu}\bar r(x, \mu_{n|\Xi}^t).
$$

Moreover,
$$
    - <\frac{d}{dt} \pi_{n, \Xi.\xi}^t(.|x), r(x,. ,\mu_{n|\Xi}^t)>\mu_{n|\Xi}^t(x)
    =
    -\frac{1}{t}r^{\pi_{n, \Xi.\xi}^{\BR,t}}(x,\mu_{n|\Xi}^t)\mu_{n|\Xi}^{\BR,t}(x)+\frac{1}{t}r^{\pi_{n, \Xi.\xi}}(x,\mu_{n|\Xi}^t)\mu_{n|\Xi}^{\BR,t}(x)
$$
and
$$
    - r^{\pi_{n, \Xi.\xi}}(x, \mu_{n|\Xi}^t) \frac{d}{dt} \mu_{n|\Xi}^t(x)
    =
    \frac{1}{t}r^{\pi_{n, \Xi.\xi}}(x, \mu_{n|\Xi}^t)\mu_{n|\Xi}^t(x) - \frac{1}{t}r^{\pi_{n, \Xi.\xi}}(x, \mu_{n|\Xi}^t)\mu^{\BR, t}_n(x)
$$
Using the definition of exploitability together with the above remarks, we deduce:
\begin{align}
    &\frac{d}{dt} \phi(\pi^t) = \frac{d}{dt}\left[\max \limits_{\pi'} J(\mu_0, \pi', \mu^{\pi}) - J(\mu_0, \pi, \mu^{\pi})\right]
    \\
    &=\sum \limits_{n=0}^{N} \,\,\sum_{\Xi, |\Xi|=n} \sum_{\xi} P(\Xi.\xi) \sum \limits_{x \in \mathcal{X}} \Big[<\nabla_{\mu}r^{\pi_{n, \Xi.\xi}^{\BR}, \xi}(x, \mu_{n|\Xi}^t), \frac{d}{dt} \mu_{n|\Xi}^t>\mu^{\BR, t}_{n, \Xi.\xi}(x) 
    \\ 
    &\qquad\qquad- <\nabla_{\mu}r^{\pi_{n, \Xi.\xi}}(x, \mu_{n|\Xi}^t, \xi), \frac{d}{dt} \mu_{n|\Xi}^t>\mu_{n|\Xi}^t(x)
    \\
    &\qquad\qquad- <\frac{d}{dt} \pi_{n, \Xi.\xi}^t(.|x), r(x,. ,\mu_{n|\Xi}^t)>\mu_{n|\Xi}^t(x) - r^{\pi_{n, \Xi.\xi}}(x, \mu_{n|\Xi}^t) \frac{d}{dt} \mu_{n|\Xi}^t(x)\Big]
    \\
    &= \sum \limits_{n=0}^{N} \,\,\sum_{\Xi, |\Xi|=n} \sum_{\xi} P(\Xi.\xi) \sum \limits_{x \in \mathcal{X}} \Big[t <\nabla_{\mu}\bar r(x, \mu_{n|\Xi}^t)), \frac{d}{dt} \mu_{n|\Xi}^t> \frac{1}{t}\left(\mu^{\BR, t}_{n|\Xi}(x)-\mu_{n|\Xi}^t(x)\right)\Big]
    \\
    &\qquad+ \sum \limits_{n=0}^{N} \,\,\sum_{\Xi, |\Xi|=n} \sum_{\xi} P(\Xi.\xi) \sum \limits_{x \in \mathcal{X}}  \big[\frac{1}{t}r^{\pi_{n, \Xi.\xi}}(x, \mu_{n|\Xi}^t)\mu_{n|\Xi}^t(x) - \frac{1}{t}r^{\pi_{n, \Xi.\xi}^{\BR,t}}(x,\mu_{n|\Xi}^t)\mu_{n|\Xi}^{\BR, t}(x)\big]
    \\
    &=  
    - \frac{1}{t} \phi(\pi^t) + \sum \limits_{n=0}^{N} \,\,\sum_{\Xi, |\Xi|=n} \sum_{\xi} P(\Xi.\xi) \sum \limits_{x \in \mathcal{X}} \big[t <\nabla_{\mu}\bar r(x, \mu_{n|\Xi}^t), \frac{d}{dt} \mu_{n|\Xi}^t> \frac{d}{dt} \mu_{n|\Xi}^t(x)\big],
\end{align}
where the last term is non-positive by  Property~\ref{prop:mono-r-dmu} (\textit{i.e.}, thanks to the monotonicity assumption).
\end{proof}

\subsection*{Experiments: A More Complex Setting for the  Beach Bar Process with common noise}

\textbf{Environment:} Following the first setting of the paper where the bar could only close at one given time step, we now introduce a second more complex setting, bringing also of common noise in the beach bar process. Namely, the bar has a probability $p$ to close at every time step up to a point (in practice, this point is half of the horizon: $\frac{N}{2}$). Once the bar is closed, it does not open again. This setting gives $\frac{N}{2} + 1$ possible realizations of the common noise: (1) the case where the bar never closes and (2) the $\frac{N}{2}$ cases where it closes at any of the first $\frac{N}{2}$ time steps.
For the sake of clarity, we only present the evolution of the distributions when the bar finally remains open after $\frac{N}{2}$ time steps, and when it closes at the $\frac{N}{2}^\text{th}$ time step.
    
    \begin{figure}[htbp]
    \centering
    \begin{subfigure}{0.33\textwidth}
      \centering
      \includegraphics[width=1.0\linewidth]{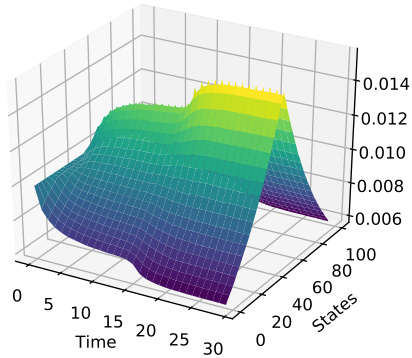}
      \caption{Model-based, the bar stays open}
    \end{subfigure}%
    \begin{subfigure}{0.33\textwidth}
      \centering
      \includegraphics[width=1.0\linewidth]{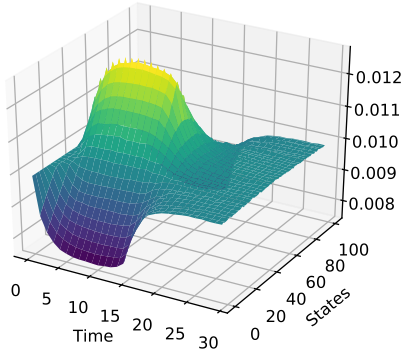}
      \caption{Model-based, the bar closes}
    \end{subfigure}
    \begin{subfigure}{0.33\textwidth}
      \centering
      \includegraphics[width=1.0\linewidth]{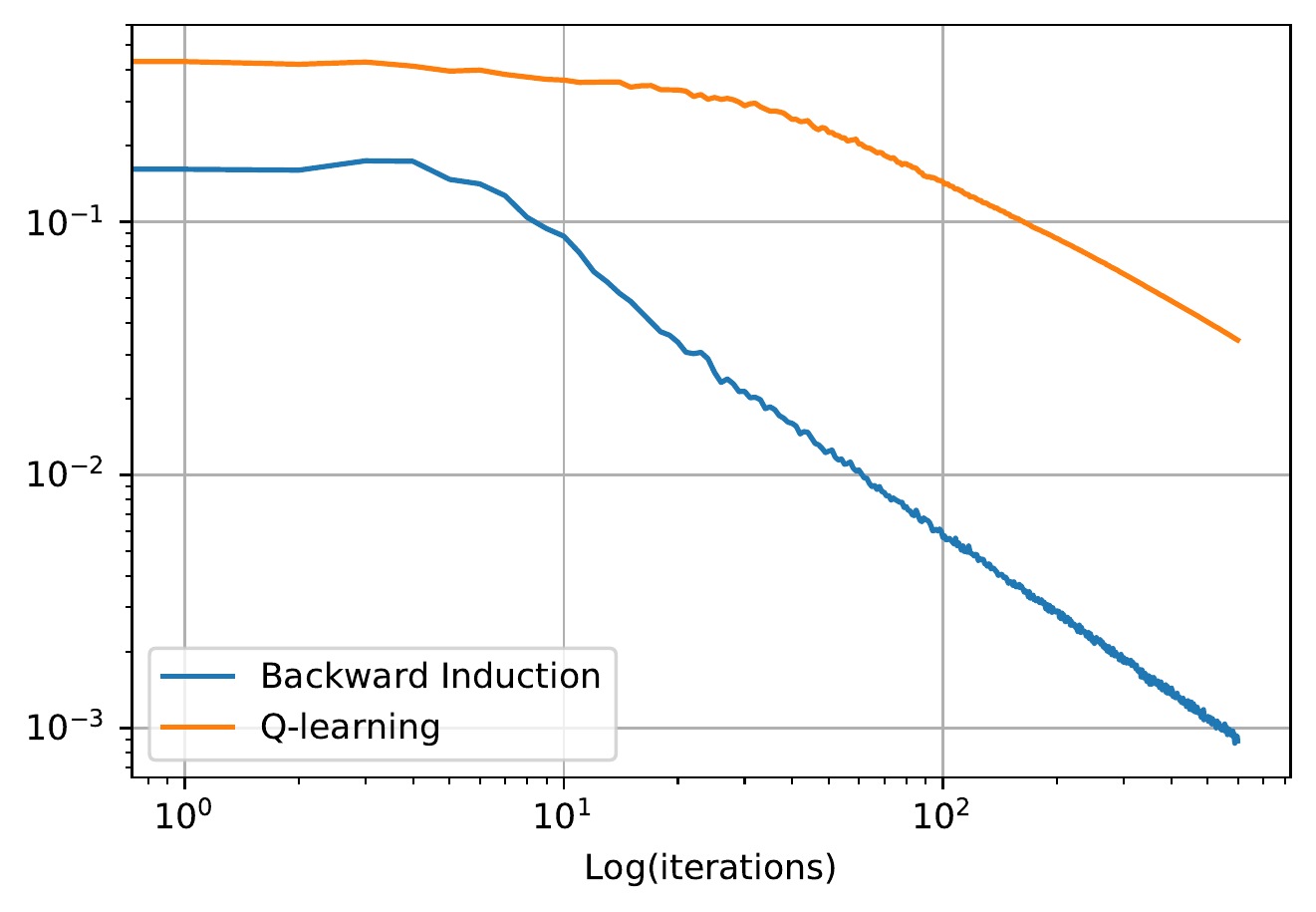}
      \caption{Exploitability}
    \end{subfigure}

    \medskip
    
    \begin{subfigure}{0.33\textwidth}
      \centering
      \includegraphics[width=1.0\linewidth]{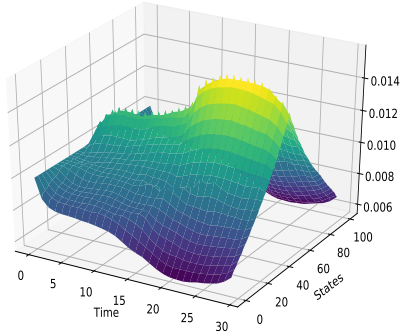}
      \caption{Model-free, the bar stays open}
    \end{subfigure}%
    \begin{subfigure}{0.33\textwidth}
      \centering
      \includegraphics[width=1.0\linewidth]{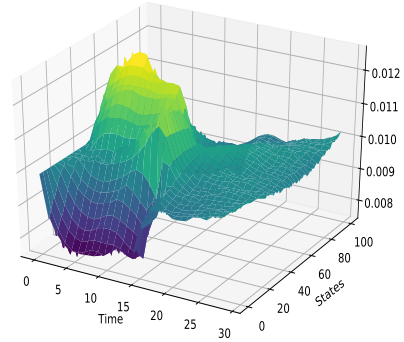}
      \caption{Model-free, the bar closes}
    \end{subfigure}
    
    \caption{2\textsuperscript{nd} common noise setting, the bar has a probability $p=0.5$ to close at every time step before $\frac{N}{2}$.}
    \label{fig:Ql-multiCN}
\end{figure}
    
\textbf{Numerical results}: Similarly to the first setting, we take $|\mathcal{X}| =100$ states and $N=30$ time steps.
As the bar has a probability $p=0.5$ to close at every time step until $\frac{N}{2}$, the distribution is flatter to anticipate the fact that people might need to spread. We can see that both model-based and model-free approaches converge to a Nash equilibrium and that model-based converges faster than model-free.

\newpage
$\;$
\newpage 
\section{Continuous Time Fictitious Play: the $\gamma$-discounted case}\label{Section_Gamma}
Surprisingly, the analysis also holds in the $\gamma$-discounted case with again the same style of reasoning. However, the distribution considered will be the $\gamma$-weighted occupancy measure instead of the distribution over states. In this section, we reintroduce the notations and we prove similar continuous time FP convergence results.

Consider, given the following:
\begin{itemize}
    \item a finite state space $\mathcal{X}$ ($x \in \mathcal{X}$), 
    \item a finite action space $\mathcal{A}$ ($a \in \mathcal{A}$),
    \item the set of distributions over state is $\Delta \mathcal{X}$ ($\mu \in \Delta \mathcal{X}$),
    \item a reward function $r(x,a,\mu)$,
    \item the transition function $p(x'|x,a)$,
    \item a policy: $\pi(a|x)$. 
\end{itemize}
We will write:
\begin{itemize}
    \item $p^\pi(x'|x) = \mathbb{E}_{a\sim \pi(.|x)}[p(x'|x,a)]$,
    \item $r^\pi(x, \mu) = \mathbb{E}_{a\sim \pi(.|x)}[r(x,a,\mu)]$,
\end{itemize}
The cumulative $\gamma$-discounted reward is defined as: $$J_{\gamma}(x_0, \pi, \mu) = \mathbb{E} \left[\sum \limits_{n=0}^{+\infty} \gamma^n r(x_n, a_n, \mu) \; | \;  x_{n+1} \sim p(.|x_n, a_n), \; a_n \sim \pi(.|x_n) \right]$$

\textbf{Useful properties:}
We have $\mu_{\gamma}^\pi (x') = \mu_0(x') + \gamma \sum \limits_{x \in \mathcal{X}} p^{\pi}(x'|x) \mu_{\gamma}^\pi (x)$ (in vectorial notations ${\mu_\gamma^{\pi}}^\top = \mu_0^\top (I-\gamma P^\pi)^{-1} $).\newline
The $\gamma$-discounted reward can be written as: $J_{\gamma}(x_0, \pi, \mu) = \sum \limits_{x \in \mathcal{X}} \mu_{\gamma}^\pi (x) r^\pi(x, \mu)$.\newline
We then have a similar formula for the policy generating the average distribution $\bar \mu^{\pi}_{\gamma}(x,t) = \frac{1}{t} \int \limits_{0}^t \mu_{\gamma}^{\pi}(x,s) ds$ can be written $\bar \pi_{\gamma}(a|x,t) = \frac{\int \limits_{0}^t \mu_{\gamma}^{\pi}(x,s) \pi(a|x,s) ds}{\int \limits_{0}^t \mu_{\gamma}^{\pi}(x,s) ds}$.

Finally, we can write: 
\begin{align}
    \bar \mu_{\gamma}^{\pi}(x,t) \bar \pi_{\gamma}(a|x,t) = \frac{1}{t} \int \limits_{0}^t \mu_{\gamma}^{\pi}(x,s) \pi(a|x,s) ds
\end{align}
And:
\begin{align}
    \dot {\bar \mu}^{\pi}_\gamma(x,t) \bar \pi_\gamma(a|x,t) + \bar \mu^{\pi}_\gamma(x,t) \dot {\bar \pi}_\gamma(a|x,t) = \frac{1}{t}\left[ \mu^{\pi}_\gamma(x,t) \pi(a|x,t)-\bar \mu^{\pi}_\gamma(x,t) \bar \pi_\gamma(a|x,t)\right].\label{dt_pin_mun}
\end{align}

\textbf{Fictitious Play in MFGs:}
In the $\gamma$-discounted case, Fictitious Play can be written as (for $t \geq 1$):
$$\dot\mu(x,t) = \frac{1}{t}(\mu^{\BR}_{\gamma}(x,t) -\mu(x,t))$$ where $\mu^{\BR}_{\gamma}(x,t)$ is the distribution of a best response against $\mu(x,t)$ of policy $\pi^{\BR}(a|x,t)$. In this section, we will write $\pi(a|x,t)$ the policy of the distribution $\mu(x,t)$. From Eq.\eqref{dt_pin_mun}, we can deduce the following property:

\begin{property}
\label{dt_pi}
$$\forall n, \;\dot \pi(a|x,t) \mu(x,t) = \frac{1}{t} \mu^{\BR}_{\gamma}(x,t) [\pi^{\BR}(a|x,t) - \pi(a|x,t)]$$
\end{property}
\begin{proof}
Such representation directly follows from   Eq.\eqref{dt_pin_mun}.
\end{proof}

We are now in position to turn to the Lyapounov congerging property of the Fictitious process. 
\begin{property}
Under the monotony assumption, we can show that the exploitability ($\phi(t) = \max \limits_{\pi'} J_{\gamma}(x_0, \pi', \mu^\pi) - J_{\gamma}(x_0, \pi, \mu^\pi)$) is a strong Lyapunov function of the system:
$$\dot \phi(t) \leq - \frac{1}{t} \phi(t)$$
\end{property}

\begin{proof}
\begin{align}
    &\dot \phi(t)\nonumber\\
    &=\sum \limits_{x \in \mathcal{X}} \big[\overbrace{<\nabla_{\mu}r^{\pi^{\BR}}(x, \mu(t)), \dot \mu(t)>\mu^{\BR}_{\gamma}(x, t) - <\nabla_{\mu}r^{\pi}(x, \mu(t)), \dot \mu(t)>\mu(x, t)}^{\textrm{With } \nabla_{\mu}r^{\pi^{\BR}}(x, \mu(t)) = \nabla_{\mu}\bar r(x, \mu(t)) \textrm{ and } \nabla_{\mu}r^{\pi}(x, \mu(t)) = \nabla_{\mu}\bar r(x, \mu(t))}\nonumber\\
    &\underbrace{- <\dot \pi(.|x,t), r(x,. ,\mu(t))>\mu(x, t)}_{=-\frac{1}{t}r^{\pi^{\BR}}(x,\mu(t))\mu_{\gamma}^{\BR}(x,t)+\frac{1}{t}r^{\pi}(x,\mu(t))\mu_{\gamma}^{\BR}(x,t)} \underbrace{- r^{\pi}(x, \mu(t)) \dot \mu(x, t)}_{=\frac{1}{t}r^{\pi}(x, \mu(t))\mu(x, t) - \frac{1}{t}r^{\pi}(x, \mu(t))\mu^{\BR}_{\gamma}(x, t)}\big]\nonumber\\
    &= \sum \limits_{x \in \mathcal{X}} \big[t <\nabla_{\mu}\bar r(x, \mu(t))), \dot \mu(t)> [\frac{1}{t}(\mu_{\gamma}^{\BR}(x, t)-\mu(x, t))]\big]\nonumber\\
    &\qquad+ \sum \limits_{x \in \mathcal{X}}  \big[\frac{1}{t}r^{\pi}(x, \mu(t))\mu(x, t) - \frac{1}{t}r^{\pi^{\BR}}(x,\mu(t))\mu_{\gamma}^{\BR}(x,t)\big]\nonumber\\
    &=  - \frac{1}{t} \phi(t) + \underbrace{ \sum \limits_{x \in \mathcal{X}} \big[t <\nabla_{\mu}\bar r(x, \mu(t))), \dot \mu(t)> \dot \mu(x, t))\big]}_{\leq 0 \textrm{ by monotony}}\\
    &\le  - \frac{1}{t} \phi(t) 
    \end{align}
\end{proof}

\subsection*{Experiment: the Beach Bar Process with $\gamma$-discounted reward.}

\textbf{Environment:} We implement the beach bar process in the $\gamma$-discounted setting.

\textbf{Numerical results:} We set $\gamma = 0.9$. The algorithm estimating the best response to a fixed distribution $\mu$ is Policy Iteration in the case of the model based approach and $Q$-learning in the model-free. As the flow of distributions converges towards the stationary distribution which is not time-dependant, we only plot the final distribution obtained after $300$ time steps (and not the evolution throughout time as before). In particular, we notice that model-based and model-free approaches converge towards the same distribution. We can also observe that the convergence rate of exploitability is $O(1/t)$ for the model-based and slower for the model-free approach.

\begin{figure}[htbp]
    \centering
    \begin{subfigure}{0.33\textwidth}
      \centering
      \includegraphics[width=1.0\linewidth]{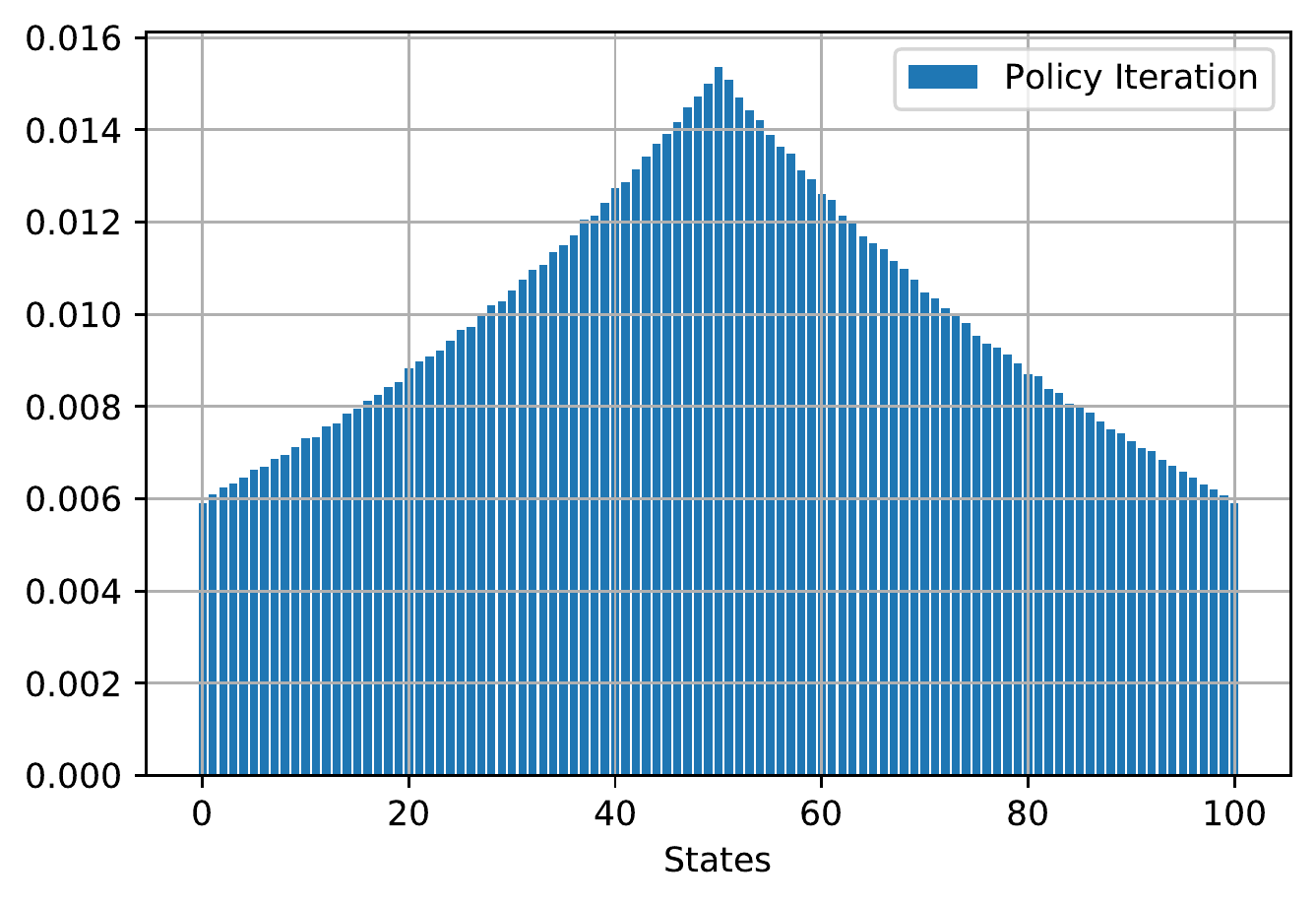}
      \caption{Model-based}
    \end{subfigure}%
    \begin{subfigure}{0.33\textwidth}
      \centering
      \includegraphics[width=1.0\linewidth]{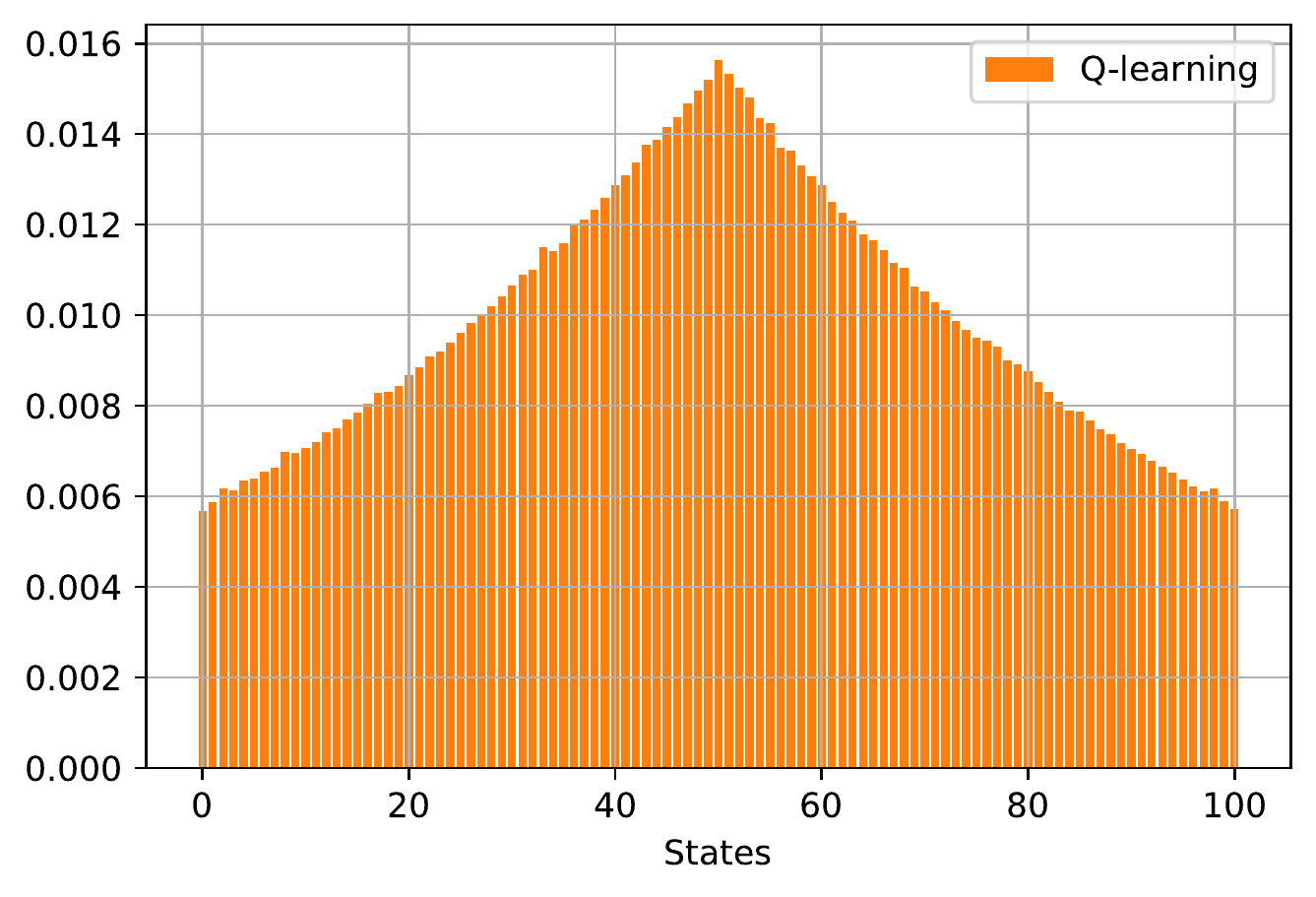}
      \caption{Model-free}
    \end{subfigure}%
    \begin{subfigure}{0.33\textwidth}
      \centering
      \includegraphics[width=1.0\linewidth]{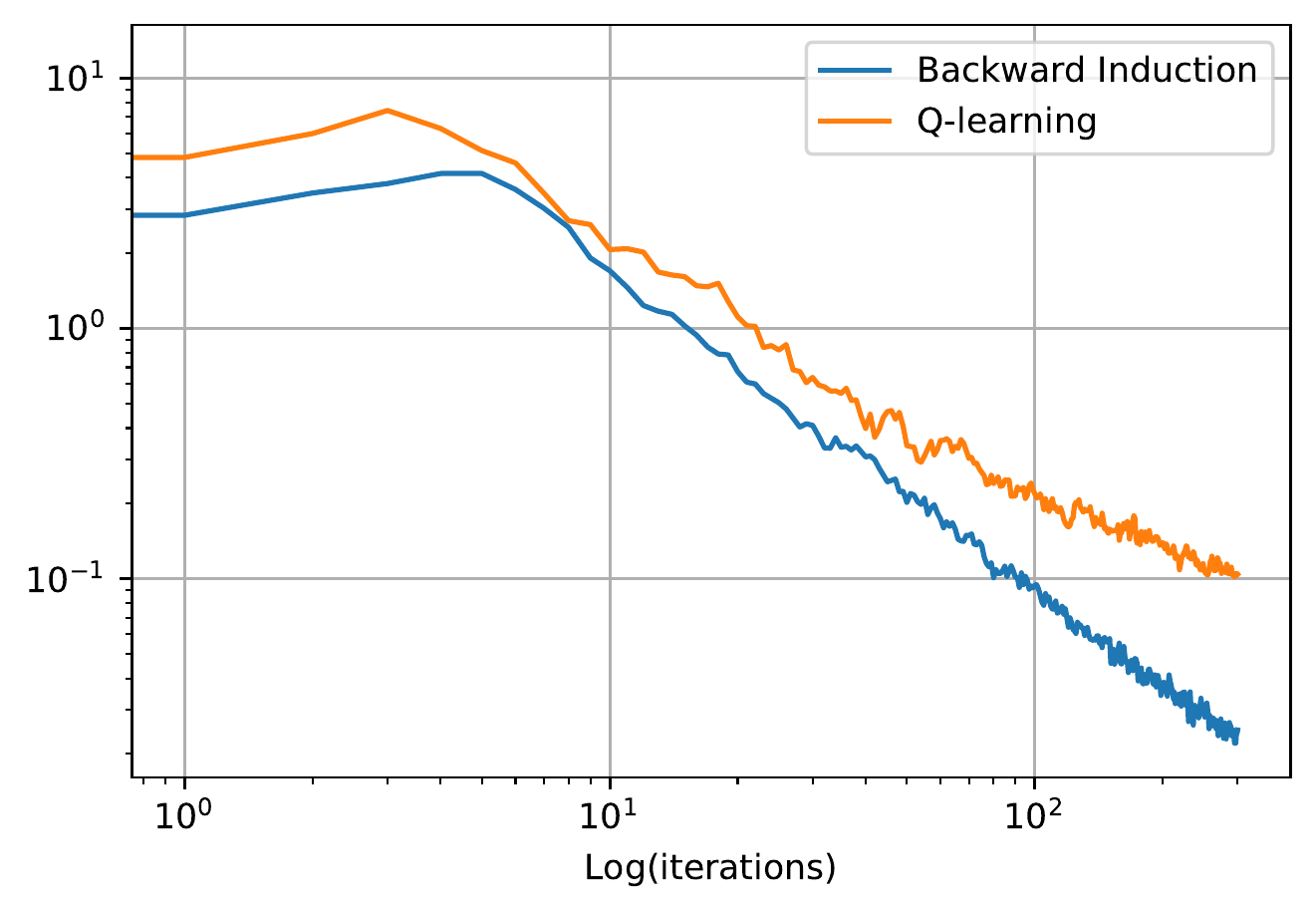}
      \caption{Exploitability}
    \end{subfigure}%
   
    \caption{Final distributions and exploitability in the $\gamma$-discounted case}
    \label{fig:distrib-exp_gamma}
    \end{figure}

\newpage    
\section{Algorithms}\label{Sec_Algorithm}

\begin{algorithm2e}[ht!]
\SetKwInOut{Input}{input}\SetKwInOut{Output}{output}
\Input{Start with a fixed distribution $\mu=(\mu_k)_k$ and $Q^k = 0$ and $\epsilon$ and the learning rate $\alpha$.}
\caption{$Q$-Learning in Mean Field Games \label{algQ_learning}}
\For{$k=0,\dots, K$:}{
    sample $x^k_0 \sim \mu_0$ \;
    \For{$n=0,\dots, N$:}{
        $a_n^k$ is $\epsilon$-greedy with respect to $Q^k(x^k_n,.)$.\;
        if not terminal sample $x_{n+1}^k$ according to $p(.|x^k_n a_n^k)$.\;
        $Q^{k+1}_n(x^k_n, a_n^k) = (1-\alpha)Q^{k+1}_n(x^k_n, a_n^k) + \alpha[r(x^k_n, a_n^k, \mu_{k-1}) + \max_b Q^{k}_{n+1}(x^k_{n+1}, b)]$\;
    }
}
\Return{$\pi^*$ a greedy policy with respect to $Q^K$}
\end{algorithm2e}
\begin{algorithm2e}[ht!]
\SetKwInOut{Input}{input}\SetKwInOut{Output}{output}
\Input{Start with a fixed policy $\pi$ and an initial distribution $\mu_0 = \mu^{\pi}_0$}
\caption{Empirical Density Estimation \label{algEDE}}
\For{$k=0,\dots, K$:}{
    sample $x^k_0 \sim \mu_0$ \;
    \For{$n=0,\dots, N$:}{
        $a_n^k$ with respect to $Q^k(x^k_n,.)$.\;
        if not terminal sample $x_{n+1}^k$ according to $p(.|x^k_n a_n^k)$.
    }
}
Finally $\forall x, n \in \mathcal{X} \times \{0, \dots, N\} \; \hat \mu^\pi_n (x) = \frac{1}{K+1} \sum_{k=0}^K 1_{x_n^k=x}$\;
\Return{$\hat \mu^\pi_n$}
\end{algorithm2e}
\begin{algorithm2e}[ht!]
\SetKwInOut{Input}{input}\SetKwInOut{Output}{output}
\Input{Start with a fixed distribution $\mu=(\mu_k)_k$ and a terminal $Q$-function $Q^{\mu}_N(x,a) = r(x, a, \mu_N)$}
\caption{Backward Induction in Mean Field Games \label{algBI}}
\For{$n=N,\dots,0$:}{
  $\pi^*_k$ is greedy with respect to $Q^{\mu}_k(x,a)$. \;
  $\forall a, x \in \mathcal{A} \times \mathcal{X} \; Q^{\mu}_{n-1}(x,a) = r(x, a, \mu_{n-1}) + \sum \limits_{x' \in \mathcal{X}}p(x'|x,a)\max \limits_b Q^{\mu}_{n}(x', b)$ \;
}
\Return{$\pi^*$}
\end{algorithm2e}
\begin{algorithm2e}[ht!]
\SetKwInOut{Input}{input}\SetKwInOut{Output}{output}
\Input{Start with a fixed policy $\pi$ and an initial distribution $\mu_0 = \mu^{\pi}_0$}
\caption{Density Estimation \label{algDE}}
\For{$n=1,\dots,N$:}{
  $\forall x \in \mathcal{X} \; \mu^{\pi}_{n}(x') = \sum \limits_{x, a \in \mathcal{X}\times \mathcal{A}} \pi_{n-1}(a|x)p(x'|x,a) \mu^{\pi}_{n-1}(x)$ \;
}
\Return{$\mu^{\pi}$}
\end{algorithm2e}
\newpage
\section{Linear Quadratic Model}\label{Sec_LQ_Appx}
\subsection{Description}
For the sake of completeness, we explain here how we obtained the benchmarck solution for the LQ problem. The original model has been introduced in~\cite{MR3325083} and corresponds to the continuous time and continuous spaces version of the LQ problem implemented in Section~\ref{sec:expe-FP-FH}. Each player can influence their speed with a control denoted by $\alpha_t$. The dynamics of the players is linear in their state, their control and the mean position, denoted by $\bar m_t$. It is affected by an idiosyncratic source of randomness $\bW = (W_t)_{t \ge 0}$ as well as a common noise in the form of a Brownian motion $\bW^0 = (W^0_t)_{t \ge 0}$. Given a flow of \emph{conditional} mean positions $\boldsymbol{\bar\mu} = (\bar \mu_t)_{t \in [0,T]}$ adapted to the filtration generated by $\bW^0$, the cost function of a representative player is defined as:
\begin{align}
	J(a; {\boldsymbol{ \bar \mu}}) 
	= \mathbb{E}\left[\int_0^T \left( \frac{1}{2}a_t^2 - q a_t (\bar \mu_t - X_t) + \frac{\kappa}{2} (\bar \mu_t - X_t)^2 \right) dt + \frac{c_{\rm term}}{2} (\bar \mu_T - X_T)^2\right]
\end{align}
Subject to the dynamics:
$$
	dX_t = [K (\bar \mu_t - X_t) + a_t] dt + \sigma \left( \rho \, dW^0_t + \sqrt{1 - \rho^2} dW_t\right).
$$
At equilibrium, we must have $\bar \mu_{t} = \mathbb{E}[X_{t}|(W^0_{s})_{s \le t}]$ for every $t \in [0,T]$.

Here, $\rho \in [0, 1]$ is a constant parameterizing the correlation between the noises, and  $q, \kappa, c, a, \sigma$ are positive constants. We assume that $q \le \kappa^2$ so that the running cost is jointly convex in the state and the control variables. 

The terms $(\bar \mu_t - X_t)$ in the dynamics and the cost function attract the process towards the mean $\bar \mu_t$. 
For the interpretation of this model in terms of systemic risk, the reader is referred to~\cite{MR3325083}. 
 The model is of linear-quadratic type and has an explicit solution through a Riccati equation, which we use as a benchmark. The optimal control at time $t$ is a linear combination of $X_t$ and $\bar \mu_t$, whose coefficients depend on time. More precisely, it is given by:
 $$
    a_t = (q+\eta_t)(\bar \mu_t - X_t),
 $$
 where $\eta$ solves the following Riccati ODE:
$$
    \dot \eta_t = 2(a+q)\eta_t + \eta_t^2 - (\kappa - q^2), \qquad \eta_T = c_{\rm term}, 
$$
whose solution is explicitly given by:
$$
    \eta_t = \frac{-(\kappa-q^2)\left(e^{(\delta^+ - \delta^-)(T-t)} - 1 \right) - c\left(\delta^+e^{(\delta^+ - \delta^-)(T-t)} - \delta^- \right)}{\left(\delta^-e^{(\delta^+ - \delta^-)(T-t)} - \delta^+ \right) - c \left(e^{(\delta^+ - \delta^-)(T-t)} - 1 \right)}
$$
where $\delta^\pm = -(a+q) \pm \sqrt{R}$ with $R = (a+q)^2 + (\kappa-q^2)>0$.

\newpage
\section{Common Success Metrics in Mean Field Games}
The optimal value function satisfies the recursive equation: 
$$
    V^{*, \mu}_N(x) = r(x, \mu_N), 
    \qquad 
    V^{*, \mu}_{n-1}(x) = \max_{a} \left\{r(x, a, \mu_{n-1}) + \sum \limits_{x' \in \mathcal{X}}p(x'|x, a)V^{*, \mu}_{n}(x') \right\}.
$$
In particular, by definition:
$$
    \max \limits_{\pi'} J(\mu_0, \pi', \mu^{\pi})
    =
    \mathbb{E}_{x \sim \mu_0}[V^{*, \mu^{\pi}}_0(x)]
$$
And:
$$
    J(\mu_0, \pi, \mu^{\pi})
    =
    \mathbb{E}_{x \sim \mu_0}[V^{\pi, \mu^{\pi}}_0(x)].
$$
Let $(x,\mu) \mapsto a^*(x,\mu)$ be such that for every $n$ and (reasonable?) $\mu$:
\begin{equation}
    \label{eq:Vstar-fixedpoint}
    V^{*, \mu}_{n-1}(x) = r(x, a^*(x,\mu_{n-1}), \mu_{n-1}) + \sum \limits_{x' \in \mathcal{X}}p(x'|x, a^*(x,\mu_{n-1}))V^{*, \mu}_{n}(x'),
\end{equation}
\textit{i.e.} $a^*$ is an optimal control. Then,  one way to check whether the value function we learned (\textit{e.g.} deduced from the $Q$-table) is a good approximate solution, is to compute the residual in the fixed point equation~\eqref{eq:Vstar-fixedpoint}. In other words, if the learned value function is $\tilde{V}$ and the policy is $\pi$ with associated distribution $\mu^{\pi}$, then, we compute:
$$
    \tilde{V}_{n-1}(x) - \left[ r(x, a^*(x,\mu^\pi_{n-1}), \mu^\pi_{n-1}) + \sum \limits_{x' \in \mathcal{X}}p(x'|x, a^*(x,\mu^\pi_{n-1}))\tilde{V}_{n}(x') \right]
$$
for every $n,x$. Taking the norm over $(n,x) \in \{1,\dots, N\} \times \mathcal{X}$ provides a metric to assess the convergence of the value function.

\textbf{Link with fixed-point iterations:} One of the most basic methods to compute a MFG equilibrium is to iteratively solve the forward equation for the distribution and the backward equation for the value function. A typical stopping criterion is that the distribution and the value function do not change too much between two successive iterations. We argue that this property implies an upper bound on the exploitability. To be specific, say that at iteration $k$, given a value function $V_{k}$ and its associated optimal control $\pi_{k}$, we compute the induced flow of distributions $\mu_{k} = \mu^{\pi_k}$, and then we compute the value function $V_{k+1}$ and the best response $\pi_{k+1}$ of an infinitesimal player against this flow of distributions. Note that:
$$
    \max_{\pi'} J(\mu_0, \pi', \mu_{k})
    =\max_{\pi'} J(\mu_0, \pi', \mu^{\pi_k})
    = J(\mu_0, \pi_{k+1}, \mu^{\pi_k})
    = \sum_{x} V_{k+1,0}(x) \mu_0(x)
$$
And:
$$
    J(\mu_0, \pi_{k}, \mu_{k}) = \sum_{x} V_{k,0}(x) \mu_0(x).
$$
Hence, if we know that $\|V_{k+1} - V_{k}\|_{\infty} := \sup_{x,n} |V_{k+1,n}(x) V_{k,n}(x)| < \epsilon$, then, in particular, $|V_{k+1,0}(x) - V_{k,0}(x)| < \epsilon$ for all $x$ and hence the exploitability is at most $\epsilon$ too.  Conversely, under suitable regularity assumptions, we can expect that a small exploitability implies $V_{k+1} \approx V_{k}$ not only at time $0$ but at every time.
\end{document}